\documentclass[12pt]{article}
\usepackage{amsmath,amsthm,amssymb,amsfonts}
\usepackage{enumerate,amscd,amsxtra}
\usepackage{mathrsfs, color}
\usepackage[cmtip,all]{xy}

 \normalsize

\newtheoremstyle{theoremstyle}
  {10pt}      
  {5pt}       
  {\itshape}  
  {}          
  {\bfseries} 
  {:}         
  {.5em}      
  {}          

\newtheoremstyle{examplestyle}
  {10pt}      
  {5pt}       
  {}          
  {}          
  {\bfseries} 
  {:}         
  {.5em}      
  {}          

\theoremstyle{theoremstyle}
\newtheorem{theorem}{Theorem}[section]
\newtheorem*{theorem*}{Theorem}
\newtheorem{lemma}[theorem]{Lemma}
\newtheorem{proposition}[theorem]{Proposition}
\newtheorem*{proposition*}{Proposition}
\newtheorem{corollary}[theorem]{Corollary}
\newtheorem*{corollary*}{Corollary}

\newtheorem{definition}[theorem]{Definition}
\newtheorem{definition*}{Definition}
\newtheorem{remark}[theorem]{Remark}
\newtheorem{remark*}{Remark}

\newcommand{\bA}{{\mathbb A}}

\newcommand{\bP}{{\mathbf P}}

\newcommand{\caG}{{\mathcal G}}

\newcommand{\caD}{{\mathcal D}}

\newcommand{\caA}{{\mathcal A}}

\newcommand{\caM}{{\mathcal M}}
\newcommand{\caN}{{\mathcal N}}

\newcommand{\caO}{{\mathcal O}}

\newcommand{\caT}{{\mathcal T}}

\newcommand{\caK}{{\mathcal K}}
\newcommand{\caQ}{{\mathcal Q}}

\newcommand{\caF}{{\mathcal F}}

\newcommand{\caU}{{\mathcal U}}
\newcommand{\caX}{{\mathfrak X}}
\newcommand{\Fh}{{\mathfrak F}}
\newcommand{\Xh}{{\mathfrak X}}
\newcommand{\Yh}{{\mathfrak Y}}
\newcommand{\Zh}{{\mathfrak Z}}

\renewcommand{\P}{{\mathbb P}}

\newcommand{\fpqc}{\mathrm{fpqc}}
\newcommand{\FG}{\underline{\mathsf{FG}}}
\newcommand{\FGL}{\mathsf{FGL}}
\newcommand{\RHom}{\underline{\mathsf{Hom}}}
\newcommand{\Ab}{\mathsf{Ab}}
\newcommand{\op}{\mathrm{op}}
\newcommand{\dual}{\vee}
\newcommand{\one}{{\bf 1}}

\newcommand{\poauf}{[\![}
\newcommand{\pozu}{] \! ]}

\newcommand{\AAA}{\mathsf{A}}
\newcommand{\BP}{\mathsf{BP}}
\newcommand{\MMM}{\mathsf{M}}
\newcommand{\NNN}{\mathsf{N}}
\newcommand{\PPP}{\mathsf{Ph}}

\newcommand{\TTT}{\mathsf{T}}
\newcommand{\Hom}{\mathsf{Hom}}
\newcommand{\Ext}{\mathsf{Ext}}

\newcommand{\EE}{\mathsf{E}}
\newcommand{\RRRR}{\mathsf{R}}
\newcommand{\Ch}{\mathsf{Ch}}
\newcommand{\PH}{\mathsf{PH}}
\newcommand{\PL}{\mathsf{PL}}
\newcommand{\PM}{\mathsf{PM}}
\newcommand{\cc}{\mathsf{c}}
\newcommand{\x}{\mathsf{x}}
\newcommand{\yy}{\mathsf{y}}
\newcommand{\s}{\mathsf{s}}
\newcommand{\PP}{\mathbf{P}}
\newcommand{\TT}{\mathcal{T}}
\newcommand{\RR}{\mathcal{R}}
\newcommand{\cO}{\mathcal{O}}
\newcommand{\MGL}{\mathsf{MGL}}
\newcommand{\MBP}{\mathsf{MBP}}

\newcommand{\Gr}{\mathbf{Gr}}
\newcommand{\RRR}{\mathbf{R}}
\newcommand{\LLL}{\mathbf{L}}

\newcommand{\Thom}{\mathsf{Th}}
\newcommand{\A}{\mathbf{A}}

\newcommand{\Top}{\mathsf{Top}}
\newcommand{\C}{\mathbf{C}}
\newcommand{\G}{\mathbf{G}_{\mathfrak{m}}}

\newcommand{\SH}{\mathbf{SH}}
\newcommand{\LL}{\mathsf{L}}

\newcommand{\BGL}{\mathsf{B}\mathbf{GL}}
\newcommand{\KGL}{\mathsf{KGL}}

\newcommand{\KU}{\mathsf{KU}}
\newcommand{\MU}{\mathsf{MU}}
\newcommand{\MZ}{\mathsf{M}\mathbf{Z}}
\newcommand{\MQ}{\mathsf{M}\mathbf{Q}}
\newcommand{\PMQ}{\mathsf{PM}\mathbf{Q}}
\newcommand{\LQ}{\mathsf{L}\mathbf{Q}}

\newcommand{\PLQ}{\mathsf{PL}\mathbf{Q}}
\newcommand{\HQ}{\mathsf{H}\mathbf{Q}}
\newcommand{\PHQ}{\mathsf{PH}\mathbf{Q}}
\newcommand{\id}{\mathsf{id}}

\newcommand{\Spec}{\mathsf{Spec}}
\newcommand{\FF}{\mathsf{F}}
\newcommand{\GG}{\mathsf{G}}
\newcommand{\Sphere}{\mathsf{S}}
\newcommand{\unit}{\mathbf{1}}
\newcommand{\Z}{\mathbf{Z}}
\newcommand{\Q}{\mathbf{Q}}
\newcommand{\OO}{\mathcal{O}}

\newcommand{\thclass}{\mathsf{th}}

\newcommand{\Qc}{\mathsf{Qc}}
\newcommand{\y}{\mathsf{y}}
\newcommand{\f}{\mathsf{f}}

\newcommand{\pp}{{\mathsf p}}
\newcommand{\qq}{{\mathsf q}}
\newcommand{\Mod}{{\mathsf {mod}}}
\newcommand{\Alg}{{\mathsf {alg}}}
\newcommand{\im}{{\mathrm{im}}}
\renewcommand{\ker}{{\mathrm{ker}}}

\title{{\bf Motivic Landweber exactness}}
\author{Niko Naumann, Markus Spitzweck, Paul Arne {\O}stv{\ae}r}
\date{\today}
\begin{document}

\maketitle
\begin{abstract}
We prove a motivic Landweber exact functor theorem. 
The main result shows the assignment given by a Landweber-type formula involving the $\MGL$-homology 
of a motivic spectrum defines a homology theory on the motivic stable homotopy category which is 
representable by a Tate spectrum. 
Using a universal coefficient spectral sequence we deduce formulas for operations of certain motivic 
Landweber exact spectra including homotopy algebraic $K$-theory. 
Finally we employ a Chern character between motivic spectra in order to compute rational algebraic 
cobordism groups over fields in terms of rational motivic cohomology groups and the Lazard ring. 
\end{abstract}
\newpage

\tableofcontents
\newpage

\section{Introduction}
The Landweber exact functor theorem combined with Brown representability provides an almost 
unreasonably efficient toolkit for constructing homotopy types out of purely algebraic data.
Among the many examples arising this way is the presheaf of elliptic homology theories on the 
moduli stack of elliptic curves.
In this paper we incite the use of such techniques in the algebro-geometric setting of motivic 
homotopy theory. 
\vspace{0.1in}

In what follows we shall state some of the main results in the paper, 
comment on the proofs and discuss some of the background and relation to previous works.
Throughout we employ a stacky viewpoint of the subject which originates with formulations in 
stable homotopy theory pioneered by Morava and Hopkins.
Let $S$ be a regular noetherian base scheme of finite Krull dimension and $\SH(S)$ the corresponding 
motivic stable homotopy category.
A complex point $\Spec(\C)\rightarrow S$ induces a functor $\SH(S)\rightarrow\SH$ to the classical 
stable homotopy category.
Much of the work in this paper is guidelined by the popular quest of hoisting results in $\SH$ to the more 
complicated motivic category. 
\vspace{0.1in}

To set the stage, 
denote by $\MGL$ the algebraic cobordism spectrum introduced by Voevodsky \cite{voevodsky-icm}.
By computation we show $(\MGL_{*},\MGL_{*}\MGL)$ is a flat Hopf algebroid in Adams  
graded abelian groups. 
(Our standard conventions concerning graded objects are detailed in Section \ref{section:3}. 
Recall that $\MGL_{*}\equiv \MGL_{2*,*}$.) 
The useful fact that $\MGL$ gives rise to an algebraic stack $[\MGL_{*}/\MGL_{*}\MGL]$ comes to bear. 
(This apparatus is reviewed in Section \ref{section:2}.)
By comparing with the complex cobordism spectrum $\MU$ we deduce a $2$-categorical commutative diagram:
\begin{equation}
\label{label:introductioncartesiandiagram}
\xymatrix{
\Spec(\MGL_{*})\ar[d]\ar[r] & \Spec(\MU_{*})\ar[d] \\
[\MGL_{*}/\MGL_{*}\MGL]\ar[r] & [\MU_{*}/\MU_{*}\MU] }
\end{equation}
The right hand part of the diagram is well-known:
Milnor's computation of $\MU_{*}$ and Quillen's identification of the canonical formal group law over 
$\MU_{*}$ with the universal formal group law are early success stories in modern algebraic topology. 
As a $\G$-stack the lower right hand corner identifies with the moduli stack of strict graded formal groups.
Our plan from the get-go was to prove (\ref{label:introductioncartesiandiagram}) is cartesian and use that 
description of the algebraic cobordism part of the diagram to deduce motivic analogs of theorems in stable 
homotopy theory.
It turns out this strategy works for general base schemes.

Recall that an $\MU_*$-module $\MMM_*$ is Landweber exact if $v^{(p)}_0,v^{(p)}_1,\dots$ forms a regular 
sequence in $\MMM_*$ for every prime $p$.
Here $v^{(p)}_0=p$ and the $v^{(p)}_i$ for $i>0$ are indecomposable elements of degree $2p^i-2$ in $\MU_*$ with 
Chern numbers divisible by $p$.
Using the cartesian diagram (\ref{label:introductioncartesiandiagram}) we show the following result 
for Landweber exact motivic homology theories, see Theorem \ref{mot-landweber} for a more precise statement. 
\begin{theorem*}
Suppose $\AAA_{*}$ is a Landweber exact graded $\MU_{*}$-algebra. 
Then 
\begin{equation*}
\MGL_{**}(-)\otimes_{\MU_{*}}\AAA_{*}
\end{equation*}
is a bigraded ring homology theory on $\SH(S)$.
\end{theorem*}
\vspace{0.1in}

Using the theorem we deduce that
\begin{equation*}
\MGL^{**}(-)\otimes_{\MU_{*}}\AAA_{*}
\end{equation*}
is a ring cohomology theory on the subcategory of strongly dualizable objects of $\SH(S)$.
In the case of the Laurent polynomial ring $\Z[\beta,\beta^{-1}]$ on the Bott element, 
this observation forms part of the proof in \cite{spitzweck-oestvaer} of the motivic Conner-Floyd 
isomorphism 
\begin{equation*}
\xymatrix{
\MGL^{**}(-)\otimes_{\MU_{*}}\Z[\beta,\beta^{-1}] \ar[r]^-{\cong} & \KGL^{**}(-) }
\end{equation*}
for the motivic spectrum $\KGL$ representing homotopy algebraic $K$-theory.
\vspace{0.1in}

Define the category of Tate objects $\SH(S)_{\TT}$ as the smallest localizing triangulated subcategory 
of the motivic stable homotopy category containing the set $\TT$ of all mixed motivic spheres 
\begin{equation*}
\Sphere^{p,q}\equiv S^{p-q}_{s}\wedge \G^{q}
\end{equation*}
of smash products of the simplicial circle $S^{1}_{s}$ and the multiplicative group scheme $\G$.
The Tate objects are precisely the cellular spectra in the terminology of \cite{dugger-isaksen}.
Our choice of wording is deeply rooted in the theory of motives.
Since the inclusion $\SH(S)_{\TT}\subseteq\SH(S)$ preserves sums and $\SH(S)$ is compactly generated, 
a general result for triangulated categories shows that it acquires a right adjoint functor 
$\pp_{\SH(S),\TT}\colon\SH(S)\rightarrow\SH(S)_{\TT}$, 
which we call the Tate projection.
When $\EE$ is a Tate object and $\FF$ a motivic spectrum there is thus an
isomorphism
\begin{equation*}
\EE_{**}(\FF)\cong 
\EE_{**}(\pp_{\SH(S),\TT}\FF).
\end{equation*}
As in topology, 
it follows that the $\EE_{**}$-homology of $\FF$ is determined by the $\EE_{**}$-homology of mixed 
motivic spheres.
This observation is a key input in showing $(\EE_{*},\EE_{*}\EE)$ is a flat Hopf algebroid in Adams graded 
abelian groups provided one - and hence both - of the canonical maps $\EE_{**}\rightarrow\EE_{**}\EE$ is 
flat and the canonical map $\EE_*\EE\otimes_{\EE_*}\EE_{**}\rightarrow\EE_{**}\EE$ is an isomorphism.
Specializing to the example of algebraic cobordism allows us to form the algebraic stack 
$[\MGL_{*}/\MGL_{*}\MGL]$ and (\ref{label:introductioncartesiandiagram}).
\vspace{0.1in}

Our motivic analog of Landweber's exact functor theorem takes the following form, see Theorem \ref{landw-thm}. 
\begin{theorem*}
Suppose $\MMM_{*}$ is an Adams graded Landweber exact $\MU_{*}$-module.
Then there exists a motivic spectrum $\EE$ in $\SH(S)_{\TT}$ and a natural isomorphism 
\begin{equation*}
\EE_{\ast\ast}(-)\cong
\MGL_{**}(-)\otimes_{\MU_{*}}\MMM_{*}
\end{equation*}
of homology theories on $\SH(S)$.
\\ \indent
In addition, 
if $\MMM_{*}$ is a graded $\MU_{*}$-algebra then $\EE$ acquires a quasi-multiplication which 
represents the ring structure on the corresponding Landweber exact theory. 
\end{theorem*}

When the base scheme is the integers $\Z$ we use motivic Landweber exactness and Voevodsky's
result that $\SH(\Z)_{\TT}$ is a Brown category \cite{voevodsky-icm},  
so that all homology theories are representable, 
to conclude the proof of the motivic exact functor theorem.
For more details and a proof of the fact that $\SH(\Z)_{\TT}$ is a Brown category we refer to 
\cite{mn-brown}. 
For a general base scheme we provide base change results which allow us to reduce to the case
of the integers.
The subcategory of Tate objects of the derived category of modules over $\MGL$ 
- relative to $\Z$ - 
turns also out to be a Brown category.
This suffices to show the above remains valid when translated verbatim to the setting of 
highly structured $\MGL$-modules.
Recall $\MGL$ is a motivic symmetric spectrum 
and the monoid axiom introduced in \cite{schwede-shipley.alg}
holds for the motivic stable structure \cite[Proposition 4.19]{Jardinestable}.
Hence the modules over $\MGL$ acquire a closed symmetric monoidal model structure.  
Moreover, 
for every cofibrant replacement of $\MGL$
in commutative motivic symmetric ring spectra
there is a Quillen equivalence between the 
corresponding module categories.
\vspace{0.1in}

We wish to emphasize the close connection between our results and the classical Landweber 
exact functor theorem.
In particular, if $\MMM_{*}$ is concentrated in even degrees
there exists a commutative ring spectrum 
$\EE^{\Top}$ in $\SH$ which represents the corresponding topological Landweber exact theory.
Although $\EE$ and $\EE^{\Top}$ are objects in widely different categories of spectra, 
it turns out there is an isomorphism
\begin{equation*}
\EE_{\ast\ast}\EE\cong
\EE_{**}\otimes_{\EE^{\Top}_{*}}\EE^{\Top}_{*}\EE^{\Top}.
\end{equation*}

In the last part of the paper we describe (co)operations and phantom maps between Landweber exact motivic spectra.
Using a universal coefficient spectral sequence we show that every $\MGL$-module $\EE$ gives rise to a surjection
\begin{equation}
\label{equation:surjection}
\xymatrix{
\EE^{p,q}(\MMM)\ar[r] &
\Hom^{p,q}_{\MGL_{**}}(\MGL_{**}\MMM,\EE_{**}), }
\end{equation}
and the kernel of (\ref{equation:surjection}) identifies with the $\Ext$-term 
\begin{equation}
\label{equation:kernel}
\Ext^{1,(p-1,q)}_{\MGL_{**}}(\MGL_{**}\MMM,\EE_{**}).
\end{equation}
Imposing the assumption that $\EE_{*}^{\Top}\EE^{\Top}$ be a projective $\EE_{*}^{\Top}$-module 
implies the given $\Ext$-term in (\ref{equation:kernel}) vanishes, 
and hence (\ref{equation:surjection}) is an isomorphism. 
The assumption on $\EE^{\Top}$ holds for unitary topological $K$-theory $\KU$ and localizations of 
Johnson-Wilson theories.
By way of example we compute the $\KGL$-cohomology of $\KGL$.
That is, 
using the completed tensor product we show there is an isomorphism of $\KGL^{**}$-algebras
\begin{equation*}
\xymatrix{
\KGL^{**}\KGL\ar[r]^-{\cong} & \KGL^{**}\widehat{\otimes}_{\KU^*}\KU^*\KU. }
\end{equation*}
By \cite{adams-clarke} the group $\KU^1\KU$ is trivial and $\KU^0\KU$ is uncountable.
We also show that $\KGL$ does not support any nontrivial phantom map.
Adopting the proof to $\SH$ reproves the analogous result for $\KU$.
The techniques we use can further be utilized to construct a Chern character in $\SH(S)$ from $\KGL$ to the 
periodized rational motivic Eilenberg-MacLane spectrum representing rational motivic cohomology.
For smooth schemes over fields we prove there is an isomorphism between rational motivic cohomology $\MQ$ 
and the Landweber spectrum representing the additive formal group law over $\Q$.
This leads to explicit computations of rational algebraic cobordism groups, 
cf.~Corollary \ref{mglrational}.
\begin{theorem*}
If $X$ is a smooth scheme over a field and $\LL^*$ denotes the (graded) Lazard ring, then there is an isomorphism
\[ 
\MGL^{**}(X)\otimes_\Z\Q
\cong
\MQ^{**}(X)\otimes_\Z\LL^*.
\]
\end{theorem*}
For finite fields it follows that $\MGL^{2**}\otimes_\Z\Q\cong\Q\otimes_\Z\LL^{*}$ and $\MGL^{**}\otimes_\Z\Q$ is 
the trivial group if $(*,*)\not\in\Z(2,1)$.
Number fields provide other examples for which $\MGL^{**}\otimes_\Z\Q$ can now be computed explicitly 
(in terms of the number of real and complex embeddings).
The theorem suggests the spectral sequence associated to the slice tower of the algebraic cobordism spectrum takes 
the expected form, 
and that it degenerates rationally, 
cf.~the works of Hopkins-Morel reviewed in \cite{LM} and Voevodsky \cite{voe-slice}.
\vspace{0.1in}

Inspired by the results herein we make some rather speculative remarks concerning future works.
The all-important chromatic approach to stable homotopy theory acquires deep interplays with the 
algebraic geometry of formal groups.
Landweber exact algebras over Hopf algebroids represent a central theme in this endeavor, 
leading for example to the bicomplete closed symmetric monoidal abelian category of $\BP_*\BP$-comodules.
The techniques in this paper furnish a corresponding Landweber exact motivic Brown-Peterson spectrum 
$\MBP$ equivalent to the constructions in \cite{hu-kriz} and \cite{vezzosi}.
The object $\MBP_*\MBP$ and questions in motivic chromatic theory at large can be investigated along 
the lines of this paper.
An exact analog of Bousfield's localization machinery in motivic stable homotopy theory was worked out 
in \cite[Appendix A]{roendigsostvar}, 
cf.~also \cite{hornbostel} for a discussion of the chromatic viewpoint.
In a separate paper \cite{nmp-nonregular} we dispense with the regularity assumption on $S$.
The results in this paper remain valid for noetherian base schemes of finite Krull dimension. 
Since this generalization uses arguments which are independent of the present work, 
we deferred it to {\em loc. cit.}
The slices of motivic Landweber spectra are studied in \cite{spitzweck} by the third author.
\vspace{0.1in}

{\bf Acknowledgments.} 
We wish to thank J.~Hornbostel, O.~R{\"o}ndigs and the referee for helpful comments on this paper.

\section{Preliminaries on algebraic stacks}
\label{section:2}
By a stack we shall mean a category fibered in groupoids over the site 
comprised by the category of commutative rings endowed with the $\fpqc$-topology. 
A stack $\Xh$ is algebraic if its diagonal is representable and affine, 
and there exists an affine scheme $U$ together with a faithfully flat map $U\to\Xh$, 
called a presentation of $\Xh$.
We refer to \cite{goerss}, 
\cite{niko} and \cite{goerssstacks} for motivation and basic properties of these notions.
\begin{lemma} 
\label{cart-lemma}
Suppose there are $2$-commutative diagrams of algebraic stacks 
\begin{equation}
\label{einsheinz} 
\xymatrix{ \Zh \ar[r]\ar[d] & \Zh'\ar[d]\\ \Xh\ar[r]& \Xh' }
\;\;\;\;\;\;
\xymatrix{ \Yh \ar[r]\ar[d]_-{\pi} & \Yh'\ar[d]\\ \Xh\ar[r]& \Xh' }
\end{equation}
where $\pi$ is faithfully flat.
Then the left hand diagram in (\ref{einsheinz}) is cartesian if and only if the naturally 
induced commutative diagram
\begin{equation}
\label{zweifrei} 
\xymatrix{ \Zh\times_{\Xh}\Yh \ar[r]\ar[d] & \Zh'\times_{\Xh'}\Yh' \ar[d]\\ \Yh\ar[r]& \Yh'}
\end{equation}
is cartesian.
\end{lemma}
\begin{proof} 
The base change of the canonical $1$-morphism $\mathfrak{c}:\Zh\to\Zh'\times_{\Xh'}\Xh$ over $\Xh$
along $\pi$ identifies with the canonically induced $1$-morphism
\[ 
\xymatrix{
\Zh\times_{\Xh}\Yh\ar[r]^-{\mathfrak{c}\times 1} &
(\Zh'\times_{\Xh'}\Xh)\times_{\Xh}\Yh\cong\Zh'\times_{\Xh'}\Yh
\cong (\Zh'\times_{\Xh'}\Yh')\times_{\Yh'}\Yh. }
\]
This is an isomorphism provided (\ref{zweifrei}) is cartesian; hence so is $\mathfrak{c}\times 1$. 
By faithful flatness of $\mathfrak{\pi}$ it follows that $\mathfrak{c}$ is an isomorphism.
The reverse implication holds trivially.
\end{proof}
\begin{corollary} 
\label{cart-corollary} 
Suppose $\Xh$ and $\Yh$ are algebraic stacks, $U\to\Xh$ and $V\to\Yh$ are presentations and there is a 
$2$-commutative diagram:
\begin{equation}
\label{drei}
\xymatrix{
U\ar[r]\ar[d] & V\ar[d]\\ 
\Xh\ar[r] & \Yh }
\end{equation}
Then (\ref{drei}) is cartesian if and only if one - and hence both - of the commutative diagrams ($i=1,2$) 
\begin{equation}
\xymatrix{
U\times_{\Xh}U \ar[r]\ar[d]_-{pr_i} & V\times_{\Yh}V\ar[d]^-{pr_i}\\ 
U\ar[r] & V}
\end{equation}
is cartesian.
\end{corollary}
\begin{proof} 
Follows from Lemma \ref{cart-lemma} since presentations are faithfully flat.
\end{proof}

A presentation $U\to{\mathfrak X}$ yields a Hopf algebroid or cogroupoid object in commutative rings
$(\Gamma(\cO_U),\Gamma(\cO_{U\times_{{\mathfrak X}}U}))$. 
Conversely, 
if $(A,B)$ is a flat Hopf algebroid, 
denote by $[\Spec(A)/\Spec(B)]$ the associated algebraic stack.
We note that by \cite[Theorem 8]{niko} there is an equivalence of $2$-categories between flat Hopf algebroids 
and presentations of algebraic stacks.

Let $\Qc_{\caX}$ denote the category of quasi-coherent $\caO_{\caX}$-modules and $\caA\in\Qc_{\caX}$ a monoid, 
or quasi-coherent sheaf of $\caO_{\caX}$-algebras. 
If $X_0$ is a scheme and $\pi:X_0\to\caX$ faithfully flat, 
then $\caA$ is equivalent to the datum of the $\caO_{X_0}$-algebra $\caA(X_0)\equiv\pi^*\caA$ combined with 
a descent datum with respect to $\xymatrix{X_1\equiv X_0\times_{\caX} X_0 \ar@<0.5ex>[r] \ar@<-0.5ex>[r] &  X_0}$.
When $X_0=\Spec(A)$ is affine, $X_1=\Spec(\Gamma)$ is affine, 
$(A,\Gamma)$ a flat Hopf algebroid and $\caA(X_0)$ a $\Gamma$-comodule algebra.

Denote the adjunction between left $\caA$-modules in $\Qc_{\caX}$ and left $\caA(X_0)$-modules in $\Qc_{X_0}$ by:
\[ 
\xymatrix{ \pi^*:\caA-\Mod \ar@<0.5ex>[r]  & \ar@<0.5ex>[l] \caA(X_0)-\Mod:\pi_* } 
\]
Since $\pi_*$ has an exact left adjoint $\pi^*$ it preserves injectives and there are isomorphisms  
\begin{equation}
\label{eins}
\Ext^n_{\caA}(\caM,\pi_*\caN)\cong
\Ext^n_{\caA(X_0)}(\pi^*\caM,\caN)
\end{equation}
between $\Ext$-groups in the categories of quasi-coherent left $\caA$- and $\caA(X_0)$-modules.

Now assume that $i:\caU\hookrightarrow\caX$ is the inclusion of an open algebraic substack.
Then \cite[Propositions 20, 22]{niko} imply $i_*:\Qc_\caU\hookrightarrow\Qc_\caX$
is an embedding of a thick subcategory; 
see also \cite[section 3.4]{niko} for a discussion of the functoriality of $\Qc_\caX$ with respect to $\caX$.
For $\caF,\caG\in\Qc_\caU$ the Yoneda description of Ext-groups gives isomorphisms
\begin{equation}
\label{zwei}
\Ext^n_\caA(\caA\otimes_{\caO_\caX} i_*\caF, \caA\otimes_{\caO_\caX} i_*\caG)\cong
\Ext^n_{i_*\caA}(i^*\caA\otimes_{\caO_\caU}\caF, i^*\caA\otimes_{\caO_\caU}\caG).
\end{equation}
We shall make use of the following general result in the context in motivic homotopy theory, 
cf.~the proof of Theorem \ref{phantoms}.
\begin{proposition}
\label{ext}
Suppose there is a $2$-commutative diagram of algebraic stacks
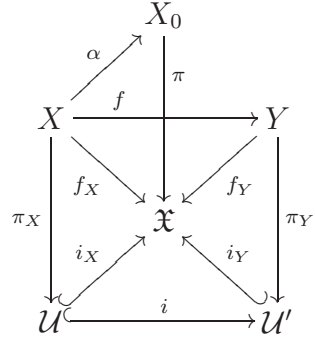
\[ 
\xymatrix{ & X_0\ar[dd]^(.3)\pi & \\ 
X\ar[ur]^\alpha \ar[rr]^(.3)f \ar[dd]_{\pi_X}\ar[rd]_{f_X}& & Y\ar[dd]^{\pi_Y}\ar[dl]^{f_Y}\\
& \caX & \\ 
\caU\ar@{^{(}->}[ur]^{i_X} \ar@{^{(}->}[rr]^i & & \caU'\ar@{_{(}->}[ul]_{i_Y}} 
\] 
where $X$, $Y$, $X_0$ are schemes, $\pi$, $\pi_X$, $\pi_Y$ faithfully flat, 
and $i_X$, $i_Y$ (hence also $i$) open inclusions of algebraic substacks.
If $\pi_Y^*\pi_{Y,*}\caO_Y\in \Qc_Y$ is projective then
\[
\Ext^n_{\caA(X_0)}(\caA(X_0)\otimes_{\caO_{X_0}}\pi^*f_{Y,*}\caO_Y,
\caA(X_0)\otimes_{\caO_{X_0}}\alpha_* \caO_X )
\]
\[ 
\cong 
\left\{ 
\begin{array}{lcc} 0 &  & n\ge 1,\\
\Hom_{\caO_Y}(\pi^*_Y\pi_{Y,*}\caO_Y, 
\caA(Y)\otimes_{\caO_Y}f_*\caO_X) &  & n=0.\end{array}\right. \]
\end{proposition}
\begin{proof} 
By (\ref{eins}) the group 
$\Ext^n_{\caA(X_0)}(\pi^*(\caA\otimes_{\caO_\caX}f_{Y,*}\caO_Y),\caA(X_0)\otimes_{\caO_{X_0}}\alpha_*\caO_X)$ 
is isomorphic to 
$\Ext^n_{\caA}(\caA\otimes_{\caO_\caX}f_{Y,*}\caO_Y,\pi_*(\pi^*\caA\otimes_{\caO_{X_0}}\alpha_*\caO_X))$, 
which the projection formula identifies with 
$\Ext^n_{\caA}(\caA\otimes_{\caO_\caX}i_{Y,*}\pi_{Y,*}\caO_Y,\caA\otimes_{\caO_\caX}i_{Y,*}i_*\pi_{X,*}\caO_X)$.
By $(\ref{zwei})$ the latter $\Ext$-group is isomorphic to 
$\Ext^n_{i_Y^*\caA}(i_Y^*\caA\otimes_{\caO_{\caU'}}\pi_{Y,*}\caO_Y,i_Y^*\caA\otimes_{\caO_{\caU'}}i_*\pi_{X,*}\caO_X)$.
Replacing $i_*\pi_{X,*}\caO_X$ by $\pi_{Y,*}f_*\caO_X$ and applying $(\ref{eins})$ gives an isomorphism to
$\Ext^n_{\caA(Y)}(\pi_Y^*(i_Y^*\caA\otimes_{\caO_{\caU'}}\pi_{Y,*}\caO_Y),\caA(Y)\otimes_{\caO_Y}f_*\caO_X)=
\Ext^n_{\caA(Y)}(\caA(Y)\otimes_{\caO_Y}\pi_Y^*\pi_{Y,*}\caO_Y,\caA(Y)\otimes_{\caO_Y}f_*\caO_X)$.
Now $\caA(Y)\otimes_{\caO_Y}\pi_Y^*\pi_{Y,*}\caO_Y$ is a projective left $\caA(Y)$-module by the assumption on 
$\pi_Y^*\pi_{Y,*}\caO_Y$.
Hence the $\Ext$-term vanishes in every positive degree, while for $n=0$, we get
\[
\Hom_{\caA(Y)}(\caA(Y)\otimes_{\caO_Y}\pi_Y^*\pi_{Y,*}\caO_Y,\caA(Y)\otimes_{\caO_Y}f_*\caO_X)
\cong
\Hom_{\caO_Y}(\pi_Y^*\pi_{Y,*}\caO_Y,\caA(Y)\otimes_{\caO_Y}f_*\caO_X).\]
\end{proof}

\section{Conventions}
\label{section:3}
The category of graded objects in an additive tensor category $\caA$ refers to integer-graded objects 
subject to the Koszul sign rule $x\otimes y=(-1)^{|x||y|}y\otimes x$.
However,
in the motivic setting, 
$\caA$ will often have a supplementary graded structure.
The category of Adams graded objects in $\caA$ refers to integer-graded objects in $\caA$, 
but no sign rule for the tensor product is introduced as a consequence of the Adams grading.
It is helpful to think of the Adams grading as being even.
We will deal with graded abelian groups,
Adams graded graded abelian groups, 
or $\Z^2$-graded abelian groups with a sign rule in the first but not in the second variable, 
and Adams graded abelian groups.
For an Adams graded graded abelian group $A_{**}$ we define $A_i\equiv A_{2i,i}$ and let $A_*$ 
denote the corresponding Adams graded abelian group. 
It will be convenient to view evenly graded $\MU_*$-modules as being Adams graded, and implicitly 
divide the grading by a factor of $2$.

The smash product induces a closed symmetric monoidal structure on $\SH(S)$.
We denote the internal function spectrum from $\EE$ to $\FF$ by $\RHom(\EE,\FF)$ and the tensor unit
or sphere spectrum by $\one$.
The Spanier-Whitehead dual of $\EE$ is by definition $\EE^\dual\equiv\RHom(\EE,\one)$. 
Note that $\EE_{**}$ with the usual indexing is an Adams graded graded abelian group.
Let $\EE_i$ be short for $\EE_{2i,i}$.
When $\EE$ is a ring spectrum, 
i.e.~a commutative monoid in $\SH(S)$, 
we implicitly assume $\EE_{**}$ is a commutative monoid in Adams graded graded abelian groups.
This latter holds true for orientable ring spectra \cite[Proposition 2.16]{hu-kriz} in view of 
\cite[Theorem 3.2.23]{MV}.

\section{Homology and cohomology theories}
An object $\FF$ of $\SH(S)$ is called finite (another term is compact) if $\Hom_{\SH(S)}(\FF,-)$ respects sums. 
Using the $5$-lemma one shows the subcategory of finite objects $\SH(S)_{\f}$ of $\SH(S)$ is thick 
\cite[Definition 1.4.3(a)]{HPS}.
For a set $\RR$ of objects in $\SH(S)_{\f}$ let $\SH(S)_{\RR,\f}$ denote the smallest thick triangulated subcategory of 
$\SH(S)_{\f}$ containing $\RR$ and $\SH(S)_\RR$ the smallest localizing subcategory of $\SH(S)$ containing $\RR$
\cite[Definition 1.4.3(b)]{HPS}. 
The examples we will deal with are the sets of mixed motivic spheres $\caT$,
the set of (isomorphism classes of) strongly dualizable objects $\caD$ 
and the set $\SH(S)_{\f}$. 

\begin{remark}
\label{comparison}
According to \cite[Remark 7.4]{dugger-isaksen} $\SH(S)_\caT\subseteq \SH(S)$ is the full subcategory of cellular 
motivic spectra introduced in loc.~cit.
\end{remark}

Recall $\FF\in\SH(S)$ is strongly dualizable if for every $\GG\in\SH(S)$ the canonical map 
\[ 
\xymatrix{
\FF^\dual \wedge \GG \ar[r] &  \RHom(\FF,\GG) }
\]
is an isomorphism. 
A strongly dualizable object is finite since $\one$ is finite.
\begin{lemma}
$\SH(S)_{\caD,\f}$ is the full subcategory of $\SH(S)_{\f}$ of strongly
dualizable objects of $\SH(S)$.
\end{lemma}
\begin{proof}
Since $\caD$ is stable under cofiber sequences and retracts, 
every object of $\SH(S)_{\caD,\f}$ is strongly dualizable. 
\end{proof}
\begin{lemma}
\label{compactobjects}
$\SH(S)_{\RR,\f}$ is the full subcategory of compact objects of $\SH(S)_\RR$ and the latter is compactly generated.
\end{lemma}
\begin{proof}
Note $\SH(S)_\RR$ is compactly generated since $\SH(S)$ is so \cite[Theorem 2.1, 2.1.1]{neeman}.
If $(-)^c$ indicates a full subcategory of compact objects \cite[Theorem 2.1, 2.1.3]{neeman} implies 
\[ 
\SH(S)_\RR^c=
\SH(S)_\RR\cap\SH(S)^c=
\SH(S)_\RR\cap\SH(S)_{\f}.
\]
Hence it suffices to show
$\SH(S)_\RR\cap\SH(S)_{\f}=
\SH(S)_{\RR,\f}$.
The inclusion ``$\supseteq$'' is obvious and to prove ``$\subseteq$'' let $\RR'$ be the smallest set of objects 
closed under suspension, retract and cofiber sequences containing $\RR$. 
Then $\RR'\subseteq\SH(S)_{\f}$ and
\[ 
\SH(S)_{\RR,\f}=\SH(S)_{\RR',\f}\subseteq\SH(S)_{\f},  
\SH(S)_\RR=\SH(S)_{\RR'}.
\]
By applying \cite[Theorem 2.1, 2.1.3]{neeman} to $\RR'$ it follows that 
\[ 
\SH(S)_\RR\cap \SH(S)_{\f}=
\SH(S)_{\RR'}\cap \SH(S)_{\f}=
\RR'\subseteq\SH(S)_{\RR',\f}=
\SH(S)_{\RR,\f}.
\]
\end{proof}
\begin{corollary}\label{exrightadj}
If $\RR\subseteq\RR'$ are as above, the inclusion $\SH(S)_\RR\subseteq\SH(S)_{\RR'}$ has a right adjoint $\pp_{\RR,\RR'}$.
\end{corollary}
\begin{proof} 
Since $\SH(S)_\RR$ is compactly generated and the inclusion preserves sums the claim follows from 
\cite[Theorem 4.1]{neeman}.
\end{proof}
\begin{definition} 
The Tate projection is the functor 
\[ 
\xymatrix{
\pp_{\SH(S)_f,\caT}:\SH(S)\ar[r] & \SH(S)_\caT. }
\] 
\end{definition}
\begin{lemma} 
\label{proj-sums}
In the situation of Corollary \ref{exrightadj}, the right adjoint $\pp_{\RR',\RR}$ preserves sums.
\end{lemma}
\begin{proof}
Using \cite[Theorem 5.1]{neeman} it suffices to show that $\SH(S)_\RR\subseteq\SH(S)_{\RR'}$ preserves 
compact objects. 
Hence by Lemma \ref{compactobjects} we are done provided $\SH(S)_{\RR,\f}\subseteq\SH(S)_{\RR',\f}$.
Clearly this holds since $\RR\subseteq \RR'$.
\end{proof}
\begin{lemma} 
\label{Tate-mod-func}
Suppose $\RR$ as above contains $\caT$. 
Then
\[ 
\xymatrix{
\pp_{\RR,\caT}:\SH(S)_\RR\ar[r] & \SH(S)_\caT }
\]
is an $\SH(S)_\caT$-module functor.
\end{lemma}
\begin{proof}
Let $\iota:\SH(S)_\caT\to\SH(S)_\RR$ be the inclusion and $\FF\in\SH(S)_\caT$, $\GG\in\SH(S)_\RR$.
Then the counit of the adjunction between $\iota$ and $\pp_{\RR,\caT}$ yields the canonical map  
\[ 
\xymatrix{
\iota(\FF\wedge \pp_{\RR,\caT}(\GG))\cong 
\iota(\FF)\wedge \iota(\pp_{\RR,\caT}(\GG))
\ar[r] & 
\iota(\FF)\wedge \GG, }
\]
adjoint to 
\begin{equation}
\label{proj} 
\xymatrix{
\FF\wedge \pp_{\RR,\caT}(\GG)\ar[r] & 
\pp_{\RR,\caT}(\iota(\FF)\wedge \GG).}
\end{equation}
We claim (\ref{proj}) is an isomorphism for all $\FF$, $\GG$.
In effect, 
the full subcategory of $\SH(S)_\caT$ generated by the objects $\FF$ for which (\ref{proj}) 
is an isomorphism for all $\GG\in\SH(S)_\RR$ is easily seen to be localizing, 
and hence we may assume $\FF=\Sphere^{p,q}$ for $p,q\in\Z$. 
The sphere $\Sphere^{p,q}$ is invertible, 
so $\SH(S)_\caT(-,\pp_{\RR,\caT}(\iota(\Sphere^{p,q})\wedge\GG))\cong\SH(S)_\RR(\iota(-),\Sphere^{p,q}\wedge\GG)$ is 
isomorphic to  
$\SH(S)_\RR(\iota(-)\wedge \Sphere^{-p,-q},\GG)\cong 
\SH(S)_\caT(-\wedge \Sphere^{-p-q}, \pp_{\RR,\caT}(\GG))\cong
\SH(S)_\caT(-, \Sphere^{p,q}\wedge \pp_{\RR,\caT}(\GG))$.
This shows $\pp_{\RR,\caT}(\iota(\Sphere^{p,q})\wedge \GG)$ and $\Sphere^{p,q}\wedge \pp_{\RR,\caT}(\GG)$ are isomorphic,
as desired.
\end{proof}

\begin{remark}
\label{cellularization}
\begin{itemize}
\item[(i)] For every $\GG\in\SH(S)$ the counit $\pp_{\RR,\caT}(\GG)\to \GG$,
where $\iota$ is omitted from the notation, 
is an $\pi_{**}$-isomorphism.
Using $\pp_{\SH(S),\caT}$ rather than the cellular functor introduced in \cite{dugger-isaksen} refines 
Proposition 7.3 of loc.~cit.
\item[(ii)] 
If $\EE\in\SH(S)_{\caT}$ and $\FF\in\SH(S)$ then $\EE_{p,q}(\FF)\cong\EE_{p,q}(\pp_{\SH(S),\caT}(\FF))$ 
on account of the isomorphisms between $\SH(S)(\Sphere^{p,q},\EE\wedge \FF)$ and 
\begin{equation*}
\SH(S)_{\caT}(\Sphere^{p,q},\pp_{\SH(S),\caT}(\EE\wedge \FF))\cong 
\SH(S)_{\caT}(\Sphere^{p,q},\EE\wedge \pp_{\SH(S),\caT}(\FF)).
\end{equation*}
In \cite{dugger-isaksen} it is argued that most spectra should be non-cellular.
On the other hand, 
the $\EE$-homology of $\FF$ agrees with the $\EE$-homology of some cellular spectrum. 
We note that many conspicuous motivic (co)homology theories are representable by cellular spectra: 
Landweber exact theories, including algebraic cobordism and homotopy algebraic $K$-theory, and also 
motivic (co)homology over fields of characteristic zero according to work of Hopkins and Morel.
\end{itemize}
\end{remark}

\begin{definition}
A homology theory on a triangulated subcategory $\TTT$ of $\SH(S)$ is a homological functor $\TTT\to\Ab$
which preserves sums. 
Dually, 
a cohomology theory on $\TTT$ is a homological functor $\TTT^\op \to \Ab$ which takes sums to products.
\end{definition}
\begin{lemma} \label{hom-clo}
Suppose $\RR\subseteq \caD$ is closed under duals.
Then every homology theory on $\SH(S)_{\RR,\f}$ extends uniquely to a homology theory on $\SH(S)_\RR$.
\end{lemma}
\begin{proof}
In view of Lemma \ref{compactobjects} we can apply \cite[Corollary 2.3.11]{HPS}
which we refer to for a more detailed discussion.
\end{proof}

Homology and cohomology theories on $\SH(S)_{\caD,\f}$ are interchangeable according to the categorical 
duality equivalence $\SH(S)_{\caD,\f}^\op \cong\SH(S)_{\caD,\f}$.
The same holds for every $\RR$ for which $\SH(S)_{\RR,\f}$ is contained in $\SH(S)_{\caD,\f}$ and closed 
under duality, 
e.g.~$\SH(S)_{\caT,\f}$.
We shall address the problem of representing homology theories on $\SH(S)$ in Section \ref{reps}.
Cohomology theories are always defined on $\SH(S)_{\f}$ unless specified to the contrary.

\begin{definition} 
Let $\TTT\subset\SH(S)$ be a triangulated subcategory closed under the smash product.
A multiplicative or ring (co)homology theory on $\TTT$, 
always understood to be commutative,
is a (co)homology theory $E$ on $\TTT$ together with maps $\Z\to E(\Sphere^{0,0})$ and
$E(\FF) \otimes E(\GG) \to E(\FF \wedge \GG)$ which are natural in $\FF,\GG \in \TTT$.
These maps are subject to the usual unitality, associativity and commutativity constraints \cite[pg.~269]{switzer}.
\end{definition}
Ring spectra in $\SH(S)$ give rise to ring homology and cohomology theories.
We shall use the following bigraded version of (co)homology theories.

\begin{definition}
Let $\TTT\subset\SH(S)$ be a triangulated subcategory closed under shifts by
all mixed motivic spheres $\Sphere^{p,q}$. 
A bigraded homology theory on $\TTT$ is a homological functor $\Phi$ from $\TTT$ to Adams graded 
graded abelian groups which preserves sums together with natural isomorphisms
$$\Phi(X)_{p,q} \cong \Phi(\Sigma^{1,0}X)_{p+1,q}$$ and
$$\Phi(X)_{p,q} \cong \Phi(\Sigma^{0,1}X)_{p,q+1}$$
such that the diagram
$$\xymatrix{
\Phi(X)_{p,q} \ar[r] \ar[d] & \Phi(\Sigma^{1,0}X)_{p+1,q} \ar[d] \\
\Phi(\Sigma^{0,1}X)_{p,q+1} \ar[r] &
\Phi(\Sigma^{1,1}X)_{p+1,q+1}}$$
commutes for all $p$ and $q$.

Bigraded cohomology theories are defined likewise.
\end{definition}

We note there is an equivalence of categories between (co)homology theories on $\TTT$ and bigraded (co)homology 
theories on $\TTT$.

\section{Tate objects and flat Hopf algebroids}
Guided by stable homotopy theory, we wish to associate flat Hopf algebroids to suitable motivic ring spectra. 
By a Hopf algebroid we shall mean a cogroupoid object in the category of commutative rings over either abelian groups, 
Adams graded abelian groups or Adams graded graded abelian groups.
Throughout this section $\EE$ is a ring spectrum in $\SH(S)_\caT$. 
We call $\EE_{**}$ flat provided one 
- and hence both - 
of the canonical maps $\EE_{**}\to\EE_{**}\EE$ is flat, and similarly for $\EE_{*}$ and $\EE_{*}\to\EE_{*}\EE$.
\begin{lemma} 
\label{flat-strmaps}
\begin{enumerate}
\item[(i)] 
If $\EE_{**}$ is flat then for every motivic spectrum $\FF$ the canonical map
$$
\xymatrix{
\EE_{**}\EE \otimes_{\EE_{**}} \EE_{**}\FF \ar[r] & (\EE \wedge \EE \wedge \FF)_{**} }
$$
is an isomorphism.
\item[(ii)]  
If $\EE_{*}$ is flat and the canonical map $\EE_* \EE \otimes_{\EE_*} \EE_{**}\to\EE_{**}\EE$
is an isomorphism, 
then for every motivic spectrum $\FF$ the canonical map
$$
\xymatrix{
\EE_*\EE \otimes_{\EE_*} \EE_* \FF  \ar[r] & (\EE \wedge \EE \wedge \FF)_* }
$$
is an isomorphism.
\end{enumerate}
\end{lemma}
\begin{proof}
$(i)$: Using Lemma \ref{Tate-mod-func} we may assume that $\FF$ is a Tate object.
The proof follows now along the same lines as in topology by first noting that the statement clearly 
holds when $\FF$ is a mixed motivic sphere, 
and secondly that we are comparing homology theories on $\SH(S)_\caT$ which respect sums.
$(ii)$: The two assumptions imply we may refer to $(i)$.
Hence there is an isomorphism
$$
\xymatrix{
\EE_{**}\EE \otimes_{\EE_{**}} \EE_{**}\FF \ar[r] &  (\EE \wedge \EE \wedge \FF)_{**}. }
$$
By the second assumption the left hand side identifies with   
$$
(\EE_*\EE \otimes_{\EE_*} \EE_{**}) \otimes_{\EE_{**}} \EE_{**} \FF\cong 
\EE_*\EE \otimes_{\EE_*} \EE_{**} \FF.
$$
Restricting to bidegrees which are multiples of $(2,1)$ yields the claimed isomorphism.
\end{proof}
\begin{corollary} 
\label{hopf-algebroids}
\begin{enumerate}
\item[(i)]  
If $\EE_{**}$ is flat then $(\EE_{**},\EE_{**} \EE)$ is canonically a flat Hopf algebroid in Adams graded graded 
abelian groups and for every $\FF\in\SH(S)$ the module $\EE_{**}\FF$ is an $(\EE_{**},\EE_{**} \EE)$-comodule.
\item[(ii)]   
If $\EE_{*}$ is flat and the canonical map $\EE_*\EE\otimes_{\EE_*}\EE_{**}\to\EE_{**}\EE$ is an isomorphism, 
then $(\EE_*, \EE_* \EE)$ is canonically a flat Hopf algebroid in Adams graded abelian groups and for every 
$\FF\in\SH(S)$ the modules $\EE_{**}\FF$ and $\EE_*\FF$ are $(\EE_*, \EE_* \EE)$-comodules.
\end{enumerate}
\end{corollary}

The second part of Corollary \ref{hopf-algebroids} is really a statement about Hopf algebroids:
\begin{lemma}
Suppose $(A_{**},\Gamma_{**})$ is a flat Hopf algebroid in Adams graded graded abelian groups and the natural
map $\Gamma_*\otimes_{A_*} A_{**}\to\Gamma_{**}$ is an isomorphism. 
Then $(A_*,\Gamma_*)$ has the natural structure of a flat Hopf algebroid in Adams graded abelian groups,  
and for every comodule $M_{**}$ over $(A_{**},\Gamma_{**})$ the modules $M_{**}$ and $M_*$ are 
$(A_*,\Gamma_*)$-comodules. 
\end{lemma}

\section{The stacks of topological and algebraic cobordism}
\subsection{The algebraic stack of $\MU$}
Denote by $\FG$ the moduli stack of one-dimensional commutative formal groups \cite{niko}. 
It is algebraic and a presentation is given by the canonical map $\FGL \to \FG$, 
where $\FGL$ is the moduli scheme of formal group laws.
The stack $\FG$ has a canonical line bundle $\omega$, 
and $[\MU_*/\MU_* \MU]$ is equivalent to the corresponding $\G$-torsor $\FG^s$ over $\FG$.

\subsection{The algebraic stack of  $\MGL$}
In this section we first study the (co)homology of finite Grassmannians over regular noetherian base schemes 
of finite Krull dimension.
Using this computational input we relate the algebraic stacks of $\MU$ and $\MGL$. 
A key result is the isomorphism
\[ 
\MGL_{**}\MGL\cong\MGL_{**}\otimes_{\MU_*}\MU_*\MU.
\]
When $S$ is a field this can easily be extracted from \cite[Theorem 5]{borghesi}.
Since it is crucial for the following, we will give a rather detailed argument for the generalization.
\vspace{0.1in}

We recall the notion of oriented motivic ring spectra formulated by Morel \cite{morel:basicproperties}, 
cf.~\cite{hu-kriz}, \cite{PPR2} and \cite{vezzosi}:
If $\EE$ is a motivic ring spectrum, 
the unit map $\unit\rightarrow\EE$ yields a class $1\in\EE^{0,0}(\unit)$ and hence by smashing with the 
projective line a class $\cc_{1}\in\EE^{2,1}(\PP^{1})$. 
An orientation on $\EE$ is a class $\cc_{\infty}\in\EE^{2,1}(\PP^{\infty})$ that restricts to $\cc_{1}$.
Note that $\KGL$ and $\MGL$ are canonically oriented.
\vspace{0.1in}

For $0\leq d\leq n$ define the ring
\begin{equation}\label{651}
\RRRR_{n,d}\equiv \Z[\x_1,\ldots,\x_{n-d}]/(\s_{d+1},\ldots,\s_n),
\end{equation}
where $\s_i$ is given by
\[ 
1+\sum\limits_{n=1}^\infty \s_nt^n\equiv (1+\x_1t+\x_2t^2+�\ldots+\x_{n-d}t^{n-d})^{-1}
\mbox{ in } \Z[\x_1,\ldots,\x_{n-d}][[t]]^{\times}.\]
By assigning weight $i$ to $\x_i$ every $\s_k\in\Z[\x_1,\ldots,\x_k]$ is homogeneous of degree $k$. 
In (\ref{651}), 
$\s_j=\s_j(\x_1,\ldots,\x_{n-d},0,\ldots)$ by convention when $d+1\leq i\leq n$.

We note that $\RRRR_{n,d}$ is a free $\Z$-module of rank $n\choose d$.
For every sequence $\underline{a}=(a_1,\ldots,a_d)$ subject to the inequalities 
$n-d\ge a_1\ge a_2\ge\ldots\ge a_d\ge 0$, 
set:
\[ 
\Delta_{\underline{a}}\equiv 
\det\left(\begin{array}{cccc} 
\x_{a_1} & \x_{a_1+1} & \ldots & \x_{a_1+d-1}\\ \x_{a_2-1} & \x_{a_2} & \ldots & \x_{a_2+d-2}\\
\ldots & \ldots & \ldots & \ldots\\ \x_{a_d-d+1} & \ldots & \ldots & \x_{a_d}\end{array}\right)
\]
Here $\x_0\equiv 1$ and $\x_i\equiv 0$ for $i<0$ or $i>n-d$.
The Schur polynomials $\{\Delta_{\underline{a}}\}$ form a basis for $\RRRR_{n,d}$ as a $\Z$-module.
Let $\pi:\RRRR_{n+1,d+1}\rightarrow \RRRR_{n,d+1}$ be the unique surjective ring homomorphism
where $\pi(\x_i)=\x_i$ for $1\leq i\leq n-d-1$ and $\pi(\x_{n-d})=0$.
It is easy to see that $\pi(\Delta_{\underline{a}})=\Delta_{\underline{a}}$
if $a_1\leq n-d-1$ and $\pi(\Delta_{\underline{a}})=0$ for $a_1=n-d$.
Hence the kernel of $\pi$ is the principal ideal generated by $\x_{n-d}$.
That is, 
\begin{equation*}
\ker(\pi)=\x_{n-d}\cdot \RRRR_{n+1,d+1}.
\end{equation*}
Moreover,
let $\iota: \RRRR_{n,d}\rightarrow\RRRR_{n+1,d+1}$ be the unique monomorphism of abelian groups 
such that for every $\underline{a}$, 
$\iota(\Delta_{\underline{a}})=\Delta_{\underline{a}'}$ 
where $\underline{a}'=(n-d,\underline{a})\equiv (n-d,a_1,\ldots,a_{d})$.
Clearly we get 
\begin{equation}
\label{65mu}
\im(\iota)=\ker(\pi).
\end{equation}
Note that $\iota$ is a map of degree $n-d$.
We will also need the unique ring homomorphism
$f:\RRRR_{n+1,d+1}\rightarrow \RRRR_{n,d}=\RRRR_{n+1,d+1}/(\s_{d+1})$
where $f(\x_i)=\x_i$ for all $1\leq i\leq n-d$. 
Elementary matrix manipulations establish the equalities 
\begin{equation}
\label{65eq1}
f(\Delta_{(a_1,\ldots,a_d,0)})=\Delta_{(a_1,\ldots,a_d)}
\end{equation}
and
\begin{equation}
\label{65eq2}
\iota(\Delta_{(a_1,\ldots,a_d)})=\x_{n-d}\cdot\Delta_{(a_1,\ldots,a_d,0)}.
\end{equation} 
Next we discuss some geometric constructions involving Grassmannians.
\vspace{0.1in}

For $0\leq d\leq n$, 
denote by $\Gr_{n-d}(\bA^n)$ the scheme parametrizing subvector bundles of rank $n-d$ of the
trivial rank $n$ bundle such that the inclusion of the subbundle is locally split.
Similarly, 
$\GG(n,d)$ denotes the scheme parametrizing locally free quotients of rank $d$ of the trivial 
bundle of rank $n$; $\GG(n,d)\cong\Gr_{n-d}(\bA^n)$ is smooth of relative dimension $d(n-d)$. 
If
\begin{equation}\label{65alpha}
\xymatrix{ 0\ar[r]& \caK_{n,d}\ar[r]& \caO_{\GG(n,d)}^n\ar[r]& \caQ_{n,d}\ar[r]& 0}
\end{equation}
is the universal short exact sequence of vector bundles on $\GG(n,d)$ and $\caK_{n,d}'$ denotes 
the dual of $\caK_{n,d}$, 
then the tangent bundle 
\begin{equation}
\label{652} 
\caT_{\GG(n,d)}\cong \caQ_{n,d}\otimes \caK_{n,d}'.
\end{equation}
The map
\[ 
\xymatrix{ i:\GG(n,d)\cong \Gr_{n-d}(\bA^n)\ar@{^{(}->}[r] &  \Gr_{n-d}(\bA^{n+1})\cong \GG(n+1,d+1)} \]
classifying $\caK_{n,d}\subseteq\caO_{\GG(n,d)}^n\hookrightarrow\caO_{\GG(n,d)}^{n+1}$
is a closed immersion.
From (\ref{652}) it follows that the normal bundle $\caN(i)$ of $i$ identifies with $\caK_{n,d}$.
Next consider the composition on $\GG(n+1,d+1)$
\[ 
\xymatrix{ 
\alpha:\caO_{\GG(n+1,d+1)}^n\ar@{^{(}->}[r] & \caO_{\GG(n+1,d+1)}^{n+1}\ar[r] & \caQ_{n+1,d+1} }
\]
for the inclusion into the first $n$ factors.
The complement of the support of coker($\alpha$) is an open subscheme $U\subseteq\GG(n+1,d+1)$ and there is a map
$\pi:U\to\GG(n,d+1)$ classifying $\alpha|_U$.
It is easy to see that $\pi$ is an affine bundle of dimension $d$, and hence
\begin{equation}
\label{65beta}
\pi\mbox{ is a motivic weak equivalence.}
\end{equation}
An argument with geometric points reveals that $U=\GG(n+1,d+1)\smallsetminus i(\GG(n,d))$.
We summarize the above with the diagram
\begin{equation}\label{653}
\xymatrix{ 
\GG(n,d) \ar@{^{(}->}[r]^-i & \GG(n+1,d+1) & U \ar@{_{(}->}[l] \ar[r]^-\pi & \GG(n,d+1).}
\end{equation}
With these precursors out of the way we are ready to compute the (co)homology of finite Grassmannians with respect
to any oriented motivic ring spectrum.

For every $0\leq d\leq n$ there is a unique morphism of $\EE^{**}$-algebras
$\varphi_{n,d}:\EE^{**}\otimes_\Z \RRRR_{n,d}\to \EE^{**}(\GG(n,d))$ such that $\varphi_{n,d}(\x_i)=\cc_i(\caK_{n,d})$ 
for $1\leq i\leq n-d$.
This follows from (\ref{65alpha}) and the standard calculus of Chern classes in $\EE$-cohomology.
Note that $\varphi_{n,d}$ is bigraded if we assign degree $(2i,i)$ to $\x_i\in\RRRR_{n,d}$.
\begin{proposition} 
For $0\leq d\leq n$ the map of $\EE^{**}$-algebras 
\[
\xymatrix{
\varphi_{n,d}: \EE^{**}\otimes_\Z \RRRR_{n,d}\ar[r] & \EE^{**}(\GG(n,d)) }
\] 
is an isomorphism.
\end{proposition}
\begin{proof} 
We observe that the result holds when $d=0$ and $d=n$, 
since then $\GG(n,d)=S$.
By induction it suffices to show that if $\varphi_{n,d}$ and $\varphi_{n,d+1}$ are isomorphisms, 
then so is $\varphi_{n+1,d+1}$.
To that end we contemplate the diagram:
\begin{equation}
\label{654}
\xymatrix{ \EE^{*-2r,*-r}(\GG(n,d)) \ar[r]^-{\alpha} & \EE^{**}(\GG(n+1,d+1)) 
\ar[r]^-{\beta} & \EE^{**}(\GG(n,d+1))\\ (\EE^{**}\otimes_\Z \RRRR_{n,d})(-2r,-r)
\ar[u]^-{\varphi_{n,d}(-2r,-r)}_-{\cong} \ar[r]^-{1\otimes\iota} & 
\EE^{**}\otimes_\Z \RRRR_{n+1,d+1} \ar[u]^-{\varphi_{n+1,d+1}} \ar[r]^-{1\otimes \pi} & 
\EE^{**}\otimes_\Z \RRRR_{n,d+1} \ar[u]^-{\varphi_{n,d+1}}_\cong}
\end{equation}
Here $r\equiv \mathrm{codim}(i)=n-d$ and $(-2r,-r)$ indicates a shift.
The top row is part of the long exact sequence in $\EE$-cohomology associated with (\ref{653}) using the 
Thom isomorphism $\EE^{*+2r,*+r}(\Thom(\caN(i)))\cong \EE^{**}(\GG(n,d))$ and the fact that 
$\EE^{**}(U)\cong \EE^{**}(\GG(n,d+1))$ by (\ref{65beta}).
The lower sequence is short exact by (\ref{65mu}).
Since $\caK_{n+1,d+1}|_U\cong\pi^*(\caK_{n,d+1})\oplus\caO_U$ we get
$\beta(\varphi_{n+1,d+1}(\x_i))=\beta(\cc_i(\caK_{n+1,d+1}))=
\cc_i(\caK_{n+1,d+1}|_U)=\pi^*(\cc_i(\caK_{n,d+1}))=\varphi_{n,d+1}(1\otimes\pi(\x_i))$.
Therefore, 
the right hand square in (\ref{654}) commutes, 
$\beta$ is surjective and the top row in (\ref{654}) is short exact.
Next we study the Gysin map $\alpha$.

Since $i^*(\caK_{n+1,d+1})=\caK_{n,d}$ there is a cartesian square of projective bundles:
\[ 
\xymatrix{ \bP(\caK_{n,d}\oplus\caO) \ar[r]^-{i'}\ar[d]_-p & 
\bP(\caK_{n+1,d+1}\oplus\caO)\ar[d]\\ \GG(n,d)\ar[r]^-i & \GG(n+1,d+1) }
\]
By the induction hypothesis $\varphi_{n,d}$ is an isomorphism. 
Thus the projective bundle theorem gives
\[ 
\EE^{**}(\bP(\caK_{n,d}\oplus\caO))\cong (\EE^{**}\otimes_\Z \RRRR_{n,d})[\x]/(\x^{r+1}+
\sum\limits_{i=1}^r(-1)^i\varphi_{n,d}(\x_i)\x^{r+1-i}),
\]
where $\x\equiv \cc_1(\caO_{\bP(\caK_{n,d}\oplus\caO)}(1))\in \EE^{2,1}(\bP(\caK_{n,d}\oplus\caO))$.
Similarly, 
\[ 
\EE^{**}(\bP(\caK_{n+1,d+1}\oplus\caO))\cong \EE^{**}(\GG(n+1,d+1))[\x']/(\x'^{r+1}+
\sum\limits_{i=1}^r(-1)^i\varphi_{n+1,d+1}(\x_i')\x'^{r+1-i}),
\]
where $\x'\equiv \cc_1(\caO_{\bP(\caK_{n+1,d+1}\oplus\caO)}(1))$ and $\x_i'=\cc_i(\caK_{n+1,d+1})\in
\RRRR_{n+1,d+1}$. 
(We denote the canonical generators of $\RRRR_{n+1,d+1}$ by $\x'_i$ in order to distinguish them 
from $\x_i\in \RRRR_{n,d}$.)
Recall the Thom class of $\caK_{n,d}\cong\caN(i)$ is constructed from
\[ 
\thclass\equiv \cc_r(p^*(\caK_{n,d})\otimes \caO_{\bP(\caK_{n,d}\oplus\caO)}(1))=
\x^r+\sum\limits_{i=1}^r (-1)^i\varphi_{n,d}(\x_i)\x^{r-i}\in \EE^{2r,r}(\bP(\caK_{n,d}\oplus\caO)).
\]
Using $i'^*(\x')=\x$ and $i^*(\varphi_{n+1,d+1}(\x_i'))=\varphi_{n,d}(\x_i)$ for $1\leq i\leq r$, 
we get that 
\[ 
\tilde{\thclass}\equiv 
\x'^r+\sum\limits_{i=1}^r(-1)^i\varphi_{n+1,d+1}(\x_i')\x'^{r-i}\in \EE^{2r,r}(\bP(\caK_{n+1,d+1}\oplus\caO))
\]
satisfies $i^{'*}(\tilde{\thclass})=\thclass$, 
and if $z:\GG(n+1,d+1)\to\bP(\caK_{n+1,d+1}\oplus\caO)$ denotes the zero-section, 
then 
\begin{equation}
\label{65eq5}
z^*(\tilde{\thclass})=(-1)^{n-d}\varphi_{n+1,d+1}(\x'_{n-d})\in \EE^{2(n-d),n-d}(\GG(n+1,d+1)).
\end{equation}
Moreover, 
since $i^*(\caK_{n+1,d+1})=\caK_{n,d}$ we conclude 
\begin{equation}
\label{65eq3}
\EE^{**}(i)\circ\varphi_{n+1,d+1}=\varphi_{n,d}\circ (1\otimes f).
\end{equation}
By inspection of the construction of the Thom isomorphism, 
it follows that 
\begin{equation}\label{65eq4}
\alpha\circ \EE^{**}(i)\mbox{ equals multiplication by } z^*(\tilde{\thclass}).
\end{equation}
And for every partition $\underline{a}$ as above,
\[ 
\alpha\circ\varphi_{n,d}(\Delta_{\underline{a}})\stackrel{(\ref{65eq1})}{=}
\alpha\circ\varphi_{n,d}\circ (1\otimes f)(\Delta_{(\underline{a},0)})\stackrel{(\ref{65eq3})}{=}
\alpha\circ \EE^{**}(i)\circ\varphi_{n+1,d+1}(\Delta_{(\underline{a},0)})
\]
\[ 
\stackrel{(\ref{65eq4})}{=}z^*(\tilde{\thclass})\cdot\varphi_{n+1,d+1}(\Delta_{(\underline{a},0)})
\stackrel{(\ref{65eq5})}{=}\varphi_{n+1,d+1}((-1)^{n-d}\x_{n-d}'\cdot\Delta_{(\underline{a},0)})
\]
\[ 
\stackrel{(\ref{65eq2})}{=}(-1)^{n-d}\cdot\varphi_{n+1,d+1}((1\otimes \iota)(\Delta_{\underline{a}})).
\]
This verifies that the left hand square in (\ref{654}) commutes up to a sign.
Hence, by the $5$-lemma, $\varphi_{n+1,d+1}$ is an isomorphism.
\end{proof}

Since $\Sigma_+^\infty\GG(n,d)\in\SH(S)$ is dualizable and $\EE$ is oriented we see that for all
$0\leq d\leq n$ the Kronecker product
\begin{equation}
\label{65dual}
\xymatrix{
\EE^{**}(\GG(n,d))\otimes_{\EE_{**}} \EE_{**}(\GG(n,d))\ar[r] & \EE_{**} }
\end{equation}
is a perfect pairing of finite free $\EE_{**}$-modules.
\begin{proposition}
\label{homofgrass}
\begin{itemize}
\item[(i)] $\EE^{**}(\BGL_d)=\EE^{**}[[\cc_1,\ldots,\cc_d]]$ where $\cc_i\in\EE^{2i,i}(\BGL_d)$ 
is the $i$th Chern class of the tautological rank $d$ vector bundle.
\item[(ii)]\begin{itemize}\item[a)] $\EE^{**}(\BGL)=\EE^{**}[[\cc_1,\cc_2,\ldots]]$ where $\cc_i$ 
is the $i$th Chern class of the universal bundle.
\item[b)] $\EE_{**}(\BGL)=\EE_{**}[\beta_0,\beta_1,\ldots]/(\beta_0=1)$ as $\EE_{**}$-algebras where 
$\beta_i\in\EE_{2i,i}(\BGL)$ is the image of the dual of $\cc_1^i \in \EE^{2i,i}(\BGL_1)$.
\end{itemize}
\item[(iii)] There are Thom isomorphisms of $\EE^{**}$-modules
\[ 
\xymatrix{
\EE^{**}(\BGL)\ar[r]^-{\cong} & \EE^{**}(\MGL) } 
\]
and $\EE_{**}$-algebras
\[ 
\xymatrix{
\EE_{**}(\MGL)\ar[r]^-{\cong} & \EE_{**}(\BGL).}
\]
\end{itemize}
\end{proposition}
\begin{proof} Parts $(i)$ and $(ii)a)$ are clear from the above. 
From (\ref{65dual}) we conclude there are canonical isomorphisms
\[ 
\xymatrix{
\EE^{**}(\BGL_d)\ar[r]^-{\cong} & \Hom_{\EE_{**}}(\EE_{**}(\BGL_d),\EE_{**}),} 
\]
\[ 
\xymatrix{
\EE_{**}(\BGL_d)\ar[r]^-{\cong} & \Hom_{\EE_{**},c}(\EE^{**}(\BGL_d),\EE_{**}). } 
\]
The notation $\Hom_{\EE_{**},c}$ refers to continuous $\EE_{**}$-linear maps with respect to the inverse limit 
topology on $\EE^{**}(\BGL_d)$ and the discrete topology on $\EE_{**}$.
Using this, 
the proofs of parts $(ii)b)$ and $(iii)$ carry over verbatim from topology.
\end{proof}
\begin{corollary} 
\label{mgl-hopf-algebroid}
\begin{enumerate}
\item[(i)] 
The tuple $(\MGL_{**}, \MGL_{**} \MGL)$ is a flat Hopf algebroid in Adams graded graded abelian groups.
For every motivic spectrum $\FF$ the module $\MGL_{**}\FF$ is an $(\MGL_{**},\MGL_{**}\MGL)$-comodule.
\item[(ii)]  
By restriction of structure the tuple $(\MGL_*, \MGL_* \MGL)$ is a flat Hopf algebroid in Adams graded abelian groups.
For every motivic spectrum $\FF$ the modules $\MGL_{**}\FF$ and $\MGL_*\FF$ are $(\MGL_*,\MGL_*\MGL)$-comodules.
\end{enumerate}
\end{corollary}
\begin{proof}
$(i)$: We note $\MGL$ is a Tate object by \cite[Theorem 6.4]{dugger-isaksen}, 
Remark \ref{comparison} and $\MGL_{**}$ is flat by Proposition \ref{homofgrass}$(iii)$ with $\EE=\MGL$.
Hence the statement follows from Corollary \ref{hopf-algebroids}$(i)$.
$(ii)$: The bidegrees of the generators $\beta_i$ in Proposition \ref{homofgrass} are multiples of $(2,1)$.
This implies the assumptions in Corollary \ref{flat-strmaps}$(ii)$ hold, and the statement follows.
\end{proof}

The flat Hopf algebroid $(\MGL_*, \MGL_* \MGL)$ gives rise to the algebraic stack
$$[\MGL_*/\MGL_*\MGL].$$
Although the grading is not required for the definition, 
it defines a $\G$-action on the stack and we may therefore form the quotient stack $[\MGL_*/\MGL_*\MGL]/\G$.
For $\FF\in\SH(S)$, let $\Fh(\FF)$ be the $\G$-equivariant quasi-coherent sheaf on $[\MGL_*/\MGL_*\MGL]$ 
associated with the comodule structure on $\MGL_*\FF$ furnished by Corollary \ref{mgl-hopf-algebroid}$(ii)$.
Denote by $\Fh/\G(\FF)$ the descended quasi-coherent sheaf on $[\MGL_*/\MGL_*\MGL]/\G$.
\begin{lemma} 
\label{2-orien-mglmgl}\begin{itemize}
\item[(i)] $\MGL_{**}\MGL\cong\MGL_{**}\otimes_{\MU_*}\MU_*\MU\cong\MGL_{**}[b_0,b_1,\ldots]/(b_0=1)$.
\item[(ii)] Let $\x$, $\x'$ be the images of the orientation on $\MGL$ with respect to the two natural maps 
$\MGL_* \to \MGL_* \MGL$.
Then $\x'=\sum_{i \ge 0} b_i \x^{i+1}$ (where $b_0=1$).
\end{itemize}
\end{lemma}
\begin{proof}
Here $b_i$ is the image under the Thom isomorphism of $\beta_i$ in Proposition \ref{homofgrass}. 
Part $(i)$ follows by comparing the familiar computation of $\MU_*\MU$ with our computation of $\MGL_{**}\MGL$.
For part $(ii)$, 
the computations leading up to \cite[Corollary 6.8]{adams} carry over to the algebraic cobordism spectrum.
\end{proof}

\subsection{Formal groups and stacks}
A graded formal group over an evenly graded ring $A_*$ or more generally over an algebraic $\G$-stack is a 
group object in formal schemes over the base with a compatible $\G$-action such that locally in the Zariski 
topology it looks like $\mathrm{Spf}(R_* \poauf \x \pozu)$,
as a formal scheme with $\G$-action,
where $\x$ has weight $-1$.  
(Note that every algebraic $\G$-stack can be covered by affine $\G$-stacks.)
This is equivalent to demanding that $\x$ has weight $0$ (or any other fixed weight) by looking at the base
change $R\to R[\yy,\yy^{-1}]$, $\yy$ of weight $1$.
A strict graded formal group is a graded formal group together with a trivialization of the line bundle
of invariant vector fields with the trivial line bundle of weight $1$.
The strict graded formal group associated with the formal group law over $\MU_*$ inherits a coaction of $\MU_* \MU$ 
compatible with the grading and the trivialization; thus, it descends to a strict graded formal group over $\FG^s$.
As a stack, 
$\FG^s$ is the moduli stack of formal groups with a trivialization of the line bundle
of invariant vector fields, 
while as a $\G$-stack it is the moduli stack of strict graded formal groups.
It follows that $\FG$ (with trivial $\G$-action) is the moduli stack of graded formal groups.
For a $\G$-stack $\Xh$ the space of $\G$-maps to $\FG$ is the space of maps from the stack quotient
$\Xh/\G$ to $\FG$. 
Hence a graded formal group is tantamount to a formal group over $\Xh/\G$.

An orientable theory gives rise to a strict graded formal group over the coefficients:
\begin{lemma} \label{hopf-act-gro}
If $\EE\in\SH(S)$ is an oriented ring spectrum satisfying the assumptions in Corollary \ref{hopf-algebroids}$(ii)$
then the corresponding strict graded formal group over $\EE_*$ inherits a compatible $\EE_* \EE$-coaction
and there is a descended strict graded formal group over the stack $[\EE_*/\EE_* \EE]$.
In particular, 
the flat Hopf algebroid $(\MGL_*, \MGL_* \MGL)$ acquires a well defined strict graded formal group, 
$[\MGL_*/\MGL_*\MGL]$ a strict graded formal group and the quotient stack $[\MGL_*/\MGL_*\MGL]/\G$ a 
formal group.
\end{lemma}
\begin{proof}
Functoriality of $\EE^*(\FF)$ in $\EE$ and $\FF$ ensures the formal group over $\EE_*$ inherits an 
$\EE_* \EE$-coaction.
For example, 
compatibility with the comultiplication of the formal group amounts to commutativity of the diagram:
$$
\xymatrix{
(\EE \wedge \EE)^*(\P^\infty) \ar[r] \ar[d] &
(\EE \wedge \EE \wedge \EE)^*(\P^\infty) \ar[d] \\
(\EE \wedge \EE)^*(\P^\infty \times \P^\infty) \ar[r] &
(\EE \wedge \EE \wedge \EE)^*(\P^\infty \times \P^\infty) }
$$
All maps respect gradings, so there is a graded formal group over the Hopf algebroid.
Different orientations yield formal group laws which differ by a strict isomorphism,  
so there is an enhanced strict graded formal group over the Hopf algebroid.
It induces a strict graded formal group over the $\G$-stack $[\MGL_*/\MGL_*\MGL]$ and quotienting out by the 
$\G$-action yields a formal group over the quotient stack.
\end{proof}

For oriented motivic ring spectra $\EE$ and $\FF$, 
denote by $\varphi(\EE,\FF)$ the strict isomorphism of formal group laws over $(\EE \wedge \FF)_*$ from the 
pushforward of the formal group law over $\EE_*$ to the one of the formal group law over $\FF_*$ given by the 
orientations on $\EE \wedge \FF$ induced by $\EE$ and $\FF$.
\begin{lemma} 
\label{change-orien}
Suppose $\EE,\FF,\GG$ are oriented spectra and let $p \colon (\EE \wedge \FF)_* \to (\EE \wedge \FF \wedge \GG)_*$,
$q \colon (\FF \wedge \GG)_* \to (\EE \wedge \FF \wedge \GG)_*$ and 
$r \colon (\EE \wedge \GG)_* \to (\EE \wedge \FF \wedge \GG)_*$
denote the natural maps.
Then $r_* \varphi(\EE,\GG) = \pp_* \varphi(\EE,\FF) \circ q_* \varphi(\FF,\GG)$.
\end{lemma}
\begin{corollary} 
\label{class-stack}
If $\EE\in\SH(S)$ is an oriented ring spectrum and satisfies the assumptions in Corollary \ref{hopf-algebroids}$(i)$, 
there is a map of Hopf algebroids $(\MU_*,\MU_*\MU) \to (\EE_{**},\EE_{**}\EE)$ such that $\MU_* \to \EE_{**}$ classifies 
the formal group law on $\EE_{**}$ and $\MU_* \MU \to \EE_{**} \EE$ the strict isomorphism $\varphi(\EE,\EE)$.
If $\EE$ satisfies the assumptions in Corollary \ref{hopf-algebroids}$(ii)$ then this map factors through a map of Hopf 
algebroids $(\MU_*,\MU_*\MU) \to (\EE_*,\EE_*\EE)$.
The induced map of stacks classifies the strict graded formal group on $[\EE_*/\EE_* \EE]$.
\end{corollary}

\subsection{A map of stacks}
Corollary \ref{class-stack} and the orientation of $\MGL$ furnish a map of flat Hopf algebroids
$$
\xymatrix{
(\MU_*, \MU_* \MU) \ar[r] & (\MGL_*, \MGL_* \MGL) }
$$
such that the induced map of $\G$-stacks $[\MGL_*/\MGL_*\MGL] \to \FG^s$
classifies the strict graded formal group on $[\MGL_*/\MGL_*\MGL]$.
Thus there is a $2$-commutative diagram:
\begin{equation} 
\label{mu-mgl-stack}
\xymatrix{
\Spec(\MGL_*) \ar[r]
\ar[d] & \Spec(\MU_*) \ar[d] \\
[\MGL_*/\MGL_*\MGL] \ar[r] & \FG^s}
\end{equation}
Quotienting out by the $\G$-action yields a map of stacks $[\MGL_*/\MGL_*\MGL]/\G \to \FG$ 
which classifies the formal group on $[\MGL_*/\MGL_*\MGL]/\G$.
\begin{proposition} 
\label{mu-mgl-cart1}
The diagram (\ref{mu-mgl-stack}) is cartesian.
\end{proposition}
\begin{proof}
Combine Corollary \ref{cart-corollary} and Lemma \ref{2-orien-mglmgl}.
Part $(ii)$ of the lemma is needed to ensure that the left and right units of $(\MU_*,\MU_*\MU)$ and
$(\MGL_*,\MGL_*\MGL)$ are suitably compatible.
\end{proof}
\begin{corollary} 
\label{mu-mgl-cart2}
The diagram
\begin{equation*} 
\label{1mu-mgl-stack}
\xymatrix{
\Spec(\MGL_*) \ar[r] \ar[d] & \Spec(\MU_*) \ar[d] \\
[\MGL_*/\MGL_*\MGL]/\G \ar[r] & \FG }
\end{equation*}
is cartesian.
\end{corollary}

\section{Landweber exact theories}
Recall the Lazard ring $\LL$ is isomorphic to $\MU_*$. 
For a prime $p$ we fix a regular sequence 
\[ 
v_0^{(p)}=p,v_1^{(p)},\ldots\in \MU_*
\]
where $v_n^{(p)}$ has degree $2(p^n-1)$ as explained in the introduction.
An (ungraded) $\LL$-module $\MMM$ is Landweber exact if $(v_0^{(p)},v_1^{(p)},\ldots)$ is a regular sequence on $\MMM$ for every $p$.
An Adams graded $\MU_*$-module $\MMM_*$ is Landweber exact if the underlying ungraded module is Landweber exact as an 
$\LL$-module \cite[Definition 2.6]{hov-strick.mor}.
In stacks this translates as follows:
An $\LL$-module $\MMM$ gives rise to a quasi-coherent sheaf $\MMM^\sim$ on $\Spec(\LL)$ and $\MMM$ is Landweber exact 
if and only if $\MMM^\sim$ is flat over $\FG$ with respect to $\Spec(\LL)\to\FG$, 
see \cite[Proposition 7]{niko}.
\begin{lemma} 
\label{grad-land}
Let $\MMM_*$ be an Adams graded $\MU_*$-module and $\MMM_*^\sim$ the associated quasi-coherent sheaf on $\Spec(\MU_*)$. 
Then $\MMM_*$ is Landweber exact if and only if $\MMM_*^\sim$ is flat over $\FG^s$ with respect to $\Spec(\MU_*)\to\FG^s$.
\end{lemma}
\begin{proof}
We need to prove the ``only if'' implication.
Assume $\MMM_*$ is Landweber exact so that $\MMM^\sim$ has a compatible $\G$-action. 
Let $\qq\colon \Spec(\MU_*) \to [\Spec(\MU_*)]/\G$ denote the quotient map and $\NNN_*^\sim$ the descended quasi-coherent sheaf 
of $\MMM_*^\sim$ on $[\Spec(\MU_*)/\G]$.
There is a canonical map $\NNN_*^\sim \to \qq_*\MMM_*^\sim$, which is the inclusion of the weight zero part of the $\G$-action.
By assumption, $\MMM_*^\sim$ is flat over $\FG$, i.e.~$\qq_*\MMM_*^\sim$ is flat over $\FG$.
Since $\NNN_*^\sim$ is a direct summand of $\qq_*\MMM_*^\sim$ it is flat over $\FG$.
Hence $\MMM_*^\sim$ is flat over $\FG^s$ since there is a cartesian diagram:
$$
\xymatrix{
\Spec(\MU_*) \ar[r] \ar[d] & \FG^s \ar[d] \\
[\Spec(\MU_*)]/\G \ar[r] & \FG }
$$
\end{proof}
\begin{remark} 
Lemma \ref{grad-land} does not hold for (ungraded) $\LL$-modules:
The map $\Spec(\Z) \to \FG^s$ classifying the strict formal multiplicative group over the integers is not flat, 
whereas the corresponding $\LL$-module $\Z$ is Landweber exact.
\end{remark}

In the following statements we view Adams graded abelian groups as Adams graded graded abelian groups via the line $\Z (2,1)$.
For example an $\MU_*$-module structure on an Adams graded graded abelian group $\MMM_{**}$ is an $\MU_*$-module in this way. 
In particular, $\MGL_{**}\FF$ is an $\MU_*$-module for every motivic spectrum $\FF$.
\begin{theorem} 
\label{mot-landweber}
Suppose $\AAA_*$ is a Landweber exact $\MU_*$-algebra,
i.e.~there is a map of commutative algebras $\MU_* \to \AAA_*$ in Adams graded abelian groups
such that $\AAA_*$ viewed as an $\MU_*$-module is Landweber exact.
Then the functor $\MGL_{**}(-) \otimes_{\MU_*} \AAA_*$ is a bigraded ring homology theory on $\SH(S)$.
\end{theorem}
\begin{proof}
By Corollary \ref{mu-mgl-cart1} there is a projection $\pp$ from 
$$
\Spec(\AAA_*) \times_{\FG^s} [\MGL_*/\MGL_*\MGL]
\cong
\Spec(\AAA_*)\times_{\Spec(\MU_*)} \Spec(\MGL_*)
$$
to $[\MGL_*/\MGL_*\MGL]$ such that 
\begin{equation} 
\label{mod-sheaf-correspondence}
\MGL_*\FF \otimes_{\MU_*}\AAA_*\cong
\Gamma(\Spec(\AAA_*) \times_{\FG^s} [\MGL_*/\MGL_*\MGL],\pp^* \Fh(\FF)).
\end{equation}
(This is an isomorphism of Adams graded abelian groups, but we won't use that fact.)
The assignment $\FF\mapsto\Fh(\FF)$ is a homological functor since $\FF\mapsto\MGL_*\FF$ is a homological functor, 
and $\pp$ is flat since it is the pullback of $\Spec(\AAA_*) \to \FG^s$ which is flat by Lemma \ref{grad-land}.
Thus $\pp^*$ is exact. 
Taking global sections over an affine scheme is an exact functor \cite[Corollary 4.23]{ueno}.
Therefore, 
$\FF \mapsto \Gamma(\Spec(\AAA_*) \times_{\FG^s} [\MGL_*/\MGL_*\MGL],\pp^* \Fh(\FF))$ is a homological functor on $\SH(S)$,
so that by (\ref{mod-sheaf-correspondence}) $\FF \mapsto \MGL_*\FF \otimes_{\MU_*} \AAA_*$ is a homological functor 
with values in Adams graded abelian groups.
It follows that $\FF \mapsto (\MGL_*\FF \otimes_{\MU_*} \AAA_*)_0$, 
the degree zero part in the Adams graded abelian group, is a homological functor, and it preserves sums.
Hence it is a homology theory on $\SH(S)$.
The associated bigraded homology theory is clearly the one formulated in the theorem. 
Finally, 
the ring structure is induced by the ring structures on the homology theory represented by $\MGL$ and on $\AAA_*$.
\end{proof}

We note the proof works using $\Fh/\G(\FF)$ instead of $\Fh(\FF)$; 
this makes the reference to Lemma \ref{grad-land} superfluous since neglecting the grading does not affect the proof. 
\begin{corollary} 
\label{mot-coh-landweber}
The functor $\MGL^{**}(-)\otimes_{\MU_*} \AAA_*$ is a ring cohomology theory on strongly dualizable motivic spectra.
\end{corollary}
\begin{proof}
Applying the functor in Theorem \ref{mot-landweber} to the Spanier-Whitehead duals of strongly dualizable motivic 
spectra yields the cohomology theory on display. 
Its ring structure is induced by the ring structure on $\AAA_*$.
\end{proof}
\begin{proposition} 
\label{affine}
The maps $[\MGL_*/\MGL_*\MGL]\to\FG^s$ and $[\MGL_*/\MGL_*\MGL]/\G\to \FG$ are affine.
\end{proposition}
\begin{proof}
Use Proposition \ref{mu-mgl-cart1}, Corollary \ref{mu-mgl-cart2} and the fact that being an affine morphism can be tested after 
faithfully flat base change.
\end{proof}
\begin{remark} 
\label{sheaf-reform}
We may formulate the above reasoning in more sheaf theoretic terms:
Namely, 
denoting by $i\colon [\MGL_*/\MGL_*\MGL] \to \FG^s$ the
canonical map, the Landweber exact theory is given by taking sections of $i_* \Fh(\FF)$ 
over $\Spec(\AAA_*)\to\FG^s$.
It is a homology theory by Proposition \ref{affine} since $\Spec(\AAA_*)\to\FG^s$ is flat.
\end{remark}

Next we give the versions of the above theorems for $\MU_*$-modules.
\begin{proposition} 
\label{landw-mod}
Suppose $\MMM_*$ is an Adams graded Landweber exact $\MU_*$-module.
Then $\MGL_{**}(-)\otimes_{\MU_*} \MMM_*$ is a homology theory on $\SH(S)$ and 
$\MGL^{**}(-) \otimes_{\MU_*} \MMM_*$ a cohomology theory on strongly dualizable spectra.
\end{proposition}
\begin{proof}
The map $i\colon [\MGL_*/\MGL_*\MGL] \to \FG^s$ is affine according to 
Proposition \ref{affine}.
With $p \colon \Spec(\MU_*) \to \FG^s$ the canonical map, 
the first functor in the proposition is given by 
$$
\xymatrix{
\FF \ar@{|->}[r] & \Gamma(\Spec(\MU_*),\MMM_* \otimes_{\MU_*} p^* i_*\Fh(\FF)), }
$$
which is exact by assumption.

The second statement is proven by taking Spanier-Whitehead duals.
\end{proof}

A Landweber exact theory refers to a homology or cohomology theory constructed as in Proposition \ref{landw-mod}.
There are periodic versions of the previous results:
\begin{proposition}
Suppose $\MMM$ is a Landweber exact $\LL$-module.
Then $\MGL_*(-)\otimes_\LL \MMM$ is a $(2,1)$-periodic homology theory on $\SH(S)$ with values in ungraded abelian groups. 
The same statement holds for cohomology of strongly dualizable objects.
These are ring theories if $\MMM$ is a commutative $\LL$-algebra.
\end{proposition}

Next we formulate the corresponding results for (highly structured) $\MGL$-modules. 
In stable homotopy theory this viewpoint is emphasized in \cite{may-idemp-landweber} and it plays an important role in this paper, 
cf.~Section \ref{op-coop}.
\begin{proposition} 
\label{mgl-landw}
Suppose $\MMM_*$ is a Landweber exact Adams graded $\MU_*$-module.
Then $\FF \mapsto \FF_{**} \otimes_{\MU_*} \MMM_*$ is a bigraded homology theory on the derived category $\caD_{\MGL}$
of $\MGL$-modules.
\end{proposition}
\begin{proof}
The proof proceeds along a now familiar route.
What follows reviews the main steps. 
We wish to construct a homological functor from $\caD_{\MGL}$ to quasi-coherent sheaves on $[\MGL_*/\MGL_*\MGL]$.
Our first claim is that for every $\FF \in \caD_{\MGL}$ the Adams graded $\MGL_*$-module $\FF_*$ is an 
$(\MGL_*, \MGL_* \MGL)$-comodule.
As in Lemma \ref{flat-strmaps}, 
$$
\xymatrix{
\MGL_{**} \MGL \otimes_{\MGL_{**}} \FF_{**}\ar[r] &  (\MGL \wedge \FF)_{**} }
$$
is an isomorphism restricting to an isomorphism
$$
\xymatrix{
\MGL_* \MGL \otimes_{\MGL_*} \FF_*\ar[r] & (\MGL \wedge \FF)_*.}
$$
This is proven by first observing that it holds for ``spheres'' $\Sigma^{p,q} \MGL$,
and secondly that both sides are homological functors which commute with sums.
This establishes the required comodule structure.
Next,
the proof of Proposition \ref{landw-mod} using flatness of $\MMM_*$ viewed as a quasi-coherent sheaf on 
$[\MGL_*/\MGL_*\MGL]$ shows the functor in question is a homology theory.
The remaining parts are clear.
\end{proof}
\begin{remark}
We leave the straightforward formulations of the cohomology, algebra and periodic versions of 
Proposition \ref{mgl-landw} to the reader.
\end{remark}

\section{Representability and base change} 
\label{reps}
Here we deal with the question when a motivic (co)homology theory is representable.
Let $\RR$ be a subset of $\SH(S)_{\f}$ such that $\SH(S)_{\RR,\f}$ consists of strongly dualizable objects,
is closed under smash products and duals and contains the unit.

First, recall the notions of unital algebraic stable homotopy categories and Brown categories from
\cite[Definition 1.1.4 and next paragraph]{HPS}:
A stable homotopy category is a triangulated category equipped with sums, a compatible closed tensor product, 
a set $\caG$ of strongly dualizable objects generating the triangulated category as a localizing subcategory, 
and such that every cohomological functor is representable.
It is unital algebraic if the tensor unit is finite (thus the objects of $\caG$ are finite) and a Brown category 
if homology functors and natural transformations between them are representable.

A map between objects in a stable homotopy category is  phantom if the induced map between the 
corresponding cohomology functors on the full subcategory of finite objects is the zero map.
In case the category is unital algebraic this holds if and only if the map between the induced homology 
theories is the zero map.
\begin{lemma} 
\label{ualghomc}
The category $\SH(S)_\RR$ is a unital algebraic stable homotopy category. 
The set $\caG$ can be chosen to be (representatives of) the objects of $\SH(S)_{\RR,\f}$.
\end{lemma}
\begin{proof}
This is an immediate application of \cite[Theorem 9.1.1]{HPS}.
\end{proof}

\begin{remark} If $S=\Spec(k)$ for a field $k$ admitting
resolutions of singularities, then $\SH(S)$ itself is unital algebraic, essentially
because every smooth $k$-scheme is strongly dualizable in $\SH(S)$, cf.
\cite[Theorem 52]{modulesovermotivic}. For $S$ the spectrum of a discrete
valuation ring $R$ with quotient field $K$, $U_+:=\Spec(K)_+\in\SH(S)$
is compact but not strongly dualizable, hence by \cite[Theorem 2.1.3,d)]{HPS}
$\SH(S)$ is not unital algebraic. We sketch a proof of the fact that $U_+$
is not dualizable which arose in discussion with J. Riou: Assume $U_+$ was dualizable, and
consider the trace of its identity, an element of $\pi_{0,0}(\one_R)$ which restricts to
$1\in \pi_{0,0}(\one_K)$ and to $0\in \pi_{0,0}(\one_\kappa)$ ($\kappa$ the residue field of $R$).
To obtain a contradiction, it would thus suffice to know that $\pi_{0,0}(\one_R)$ is simple,
which seems plausible but is open to the authors' knowledge. However, it suffices to construct
a tensor-functor $\SH(S)\to D$ (a ``realization'') such that the corresponding statements hold in 
$D$. Taking for $D$ the category of $\LQ$-modules (cf. Section \ref{CHERN}) is easily seen to work.
\end{remark}

\begin{lemma} 
\label{brown}
Suppose $S$ is covered by Zariski spectra of countable rings. 
Then $\SH(S)_\RR$ is a Brown category and the category of homology functors on $\SH(S)_\RR$ is naturally equivalent 
to $\SH(S)_\RR$ modulo phantom maps.
\end{lemma}
\begin{proof}
The first part follows by combining \cite[Theorem 4.1.5]{HPS} and \cite[Proposition 5.5]{voevodsky-icm}, \cite[Theorem 1]{mn-brown}
 and the second part
by the definition of a Brown category.
\end{proof}

Suppose $\RR,\RR'$ are as above and $\SH(S)_{\RR,\f}\subset\SH(S)_{\RR',\f}$. 
Then a cohomology theory on $\SH(S)_{\RR',\f}$ represented by $\FF$ restricts to a cohomology theory on $\SH(S)_{\RR,\f}$ 
represented by $\pp_{\RR',\RR}(\FF)$.
For Landweber exact theories the following holds:
\begin{proposition}
\label{extension}
Suppose a Landweber exact homology theory restricted to $\SH(S)_{\caT,\f}$ is represented by a Tate spectrum $\EE$.
Then $\EE$ represents the theory on $\SH(S)$.
\end{proposition}
\begin{proof}
Let $\MMM_*$ be a Landweber exact Adams graded $\MU_*$-module affording the homology theory under consideration.
By assumption there is an isomorphism on $\SH(S)_{\caT,\f}$
\[ 
\EE_{**}(-)\cong \MGL_{**}(-)\otimes_{\MU_*} \MMM_*.
\]
By Lemma \ref{hom-clo} the isomorphism extends to $\SH(S)_{\caT}$. 
Since $\MGL$ is cellular, 
an argument as in Remark \ref{cellularization} shows that both sides of the isomorphism remain unchanged when 
replacing a motivic spectrum by its Tate projection.
\end{proof}

Next we consider a map $f\colon S' \to S$ of base schemes.
The derived functor $\LLL f^*$, 
see \cite[Proposition A.7.4]{PPR1}, 
sends the class of compact generators $\Sigma^{p,q}\Sigma^\infty X_+$ of $\SH(S)$ 
- $X$ a smooth $S$-scheme - to compact objects of $\SH(S')$. 
Hence \cite[Theorem 5.1]{neeman} implies $\RRR f_*$ preserves sums, 
and the same result shows $\LLL f^*$ preserves compact objects in general.
A modification of the proof of Lemma \ref{Tate-mod-func} shows $\RRR f_*$ is an $\SH(S)_\caT$-module functor, 
i.e.~there is an isomorphism 
\begin{equation}
\label{projection} 
\RRR f_*(\FF'\wedge \LLL f^*\GG)\cong \RRR f_*(\FF')\wedge \GG
\end{equation}
in $\SH(S)$, which is natural in $\FF'\in\SH(S')$, $\GG\in\SH(S)_{\caT}$.
\begin{proposition} 
\label{base-change}
Suppose a Landweber exact homology theory over $S$ determined by the Adams graded $\MU_*$-module $\MMM_*$ is 
representable by $\EE\in\SH(S)_{\caT}$.
Then $\LLL f^* \EE\in\SH(S')_{\caT}$ represents the Landweber exact homology theory over $S'$ determined by $\MMM_*$.
\end{proposition}
\begin{proof}
For an object $\FF'$ of $\SH(S')$, adjointness, the assumption on $\EE$ and (\ref{projection}) imply 
$(\LLL f^*\EE)_{**}(\FF')=\pi_{**}(\FF'\wedge \LLL f^* \EE)$ is isomorphic to  
\[
\pi_{**}(\RRR f_*(\FF'\wedge \LLL f^*\EE))\cong 
\pi_{**}(\RRR f_*\FF'\wedge \EE)\cong 
\pi_{**}(\MGL\wedge \RRR f_*\FF')\otimes_{\MU_*}\MMM_*.
\]
Again by adjointness and (\ref{projection}) there is an isomorphism with 
\[
\pi_{**}(\MGL_{S'}\wedge \FF')\otimes_{\MU_*}\MMM_*=\MGL_{S',**}\FF'\otimes_{\MU_*}\MMM_*.
\]
\end{proof}

In the next lemma we show the pullback from Proposition \ref{base-change} respects multiplicative structures.
In general one cannot expect that ring structures on the homology theory lift to commutative monoid structures 
on representing spectra.
Instead we will consider quasi-multiplications on spectra, 
by which we mean maps $\EE\wedge\EE\to\EE$ rendering the relevant diagrams commutative up to phantom maps.
\begin{lemma} 
\label{base-change-quasi-m}
Suppose a Landweber exact homology theory afforded by the Adams graded $\MU_*$-algebra $\AAA_*$ is represented 
by a Tate object $\EE\in\SH(S)_\caT$ with quasi-multiplication $m\colon\EE\wedge\EE\to\EE$. 
Then $\LLL f^*m\colon \LLL f^*\EE \wedge \LLL
f^*\EE\to \LLL f^*\EE$ is a quasi-multiplication and represents the ring structure on 
the Landweber exact homology theory determined by $\AAA_*$ over $S'$.
\end{lemma}
\begin{proof}
Let $\phi \colon \FF_1 \wedge \FF_2 \to \FF_3$ be a map in $\SH(S)_{\caT}$. 
Let $\FF_i'$ be the base change of $\FF_i$ to $S'$.
If $\FF', \GG' \in \SH(S')$ there are isomorphisms $\FF_{i,**}'\FF'\cong \FF_{i,**} \RRR f_* \FF'$ employed in the 
proof of Proposition \ref{base-change}, and likewise for $\GG'$.
These isomorphisms are compatible with $\phi$ in the sense provided by the commutative diagram:
$$
\xymatrix{\FF_{1,**}' \FF' \otimes \FF_{2,**}' \GG' \ar[r] & \FF_{3,**}'(\FF'\wedge \GG') \\
& \FF_{3,**} (\RRR f_*(\FF' \wedge \GG')) \ar[u]_\cong \\
\FF_{1,**} \RRR f_* \FF' \otimes \FF_{2,**} \RRR f_* \GG' \ar[uu]_\cong \ar[r] & 
\FF_{3,**} (\RRR f_* \FF' \wedge \RRR f_* \GG') \ar[u] }
$$
Applying the above to the quasi-multiplication $m$ implies $\LLL f^* m$ represents the ring structure on the 
Landweber theory over $S'$. 
Hence $\LLL f^*m$ is a quasi-multiplication since the commutative diagrams exist for the homology theories, 
i.e. up to phantom maps.
\end{proof}

We are ready to prove the motivic analog of Landweber's exact functor theorem.
\begin{theorem} 
\label{landw-thm}
Suppose $\MMM_*$ is an Adams graded Landweber exact $\MU_*$-module. 
Then there exists a Tate object $\EE\in\SH(S)_\caT$ and an isomorphism of homology theories on $\SH(S)$
\[ 
\EE_{**}(-)\cong\MGL_{**}(-)\otimes_{\MU_*}\MMM_*.
\]
In addition, 
if $\MMM_*$ is a graded $\MU_*$-algebra, 
then $\EE$ acquires a quasi-multiplication which represents the ring structure on the Landweber exact theory.
\end{theorem}
\begin{proof}
First, let $S=\Spec(\Z)$. 
By Landweber exactness, see Proposition \ref{landw-mod},
the right hand side of the claimed isomorphism is a homology theory on $\SH(\Z)$.
Its restriction to $\SH(\Z)_{\caT,\f}$ is represented by some $\EE\in\SH(\Z)_{\caT}$ since $\SH(\Z)_{\caT}$ is a 
Brown category by Lemma \ref{brown}. 
We may conclude in this case using Proposition \ref{extension}. 
The general case follows from Proposition \ref{base-change} since $\LLL f^*(\SH(\Z)_{\caT})\subseteq\SH(S)_{\caT}$
for $f:S\to\Spec(\Z)$.

Now assume $\MMM_*$ is a graded $\MU_*$-algebra.
We claim that the representing spectrum $\EE\in\SH(\Z)_{\caT}$ has a quasi-multiplication representing 
the ring structure on the Landweber theory: 
The corresponding ring cohomology theory on $\SH(\Z)_{\caT,\f}$ can be extended to ind-representable 
presheaves on $\SH(\Z)_{\caT,\f}$. 
Evaluating $\EE(\FF) \otimes \EE(\GG) \to \EE(\FF \wedge \GG)$ with $\FF=\GG$ the ind-representable presheaf 
given by $\EE$ on $\id_\EE \otimes \id_\EE$ gives a map $(\EE \wedge \EE)_0(-) \to \EE_0(-)$ of homology theories.
Since $\SH(\Z)_{\caT}$ is a Brown category this map lifts to a map $\EE \wedge \EE \to \EE$ of spectra which 
is a quasi-multiplication since it represents the multiplication of the underlying homology theory.
The general case follows from Lemma \ref{base-change-quasi-m}.
\end{proof}
\begin{remark}
A complex point $\Spec(\C) \to S$ induces a sum preserving $\SH(S)_\caT$-module realization functor 
$r\colon \SH(S) \to \SH$ to the stable homotopy category.
By the proof of Proposition \ref{base-change} it follows that the topological realization of a Landweber exact 
theory is the corresponding topological Landweber exact theory, as one would expect.
\end{remark} 
\begin{proposition}
Suppose $\MMM_*$ is an Adams graded Landweber exact $\MU_*$-module. 
Then there exists an $\MGL$-module $\EE$ and an isomorphism 
of homology theories on $\caD_\MGL$
\[ 
(\EE\wedge_\MGL -)_{**}
\cong (-)_{**}\otimes_{\MU_*}\MMM_*.
\]
In addition, 
if $\MMM_*$ is a graded $\MU_*$-algebra then $\EE$ acquires a quasi-multiplication in $\caD_\MGL$ which represents
the ring structure on the Landweber exact theory.
\end{proposition}
\begin{proof}
We indicate a proof.
By Proposition \ref{mgl-landw} it suffices to show that the homology theory given by the right hand side
of the isomorphism is representable.
When the base scheme is $\Spec(\Z)$ we claim that $\caD_{\MGL,\caT}$ is a Brown category. 
In effect, 
$\SH(S)_{\f}$ is countable, 
cf.~\cite[Proposition 5.5]{voevodsky-icm}, \cite[Theorem 1]{mn-brown}, 
and $\MGL$ is a countable direct homotopy limit of finite spectra, 
so it follows that $\caD_{\MGL,\caT,\f}$ is also countable.
The conclusion that $\caD_{\MGL,\caT}$ be a Brown category follows now from \cite[Theorem 4.1.5]{HPS}.
Thus there exists an object of $\caD_{\MGL,\caT}$ representing the Landweber exact theory over $\Spec(\Z)$.
Now let $f\colon S\to\Spec(\Z)$ be the unique map and $\LLL f_\MGL^*\colon\caD_{\MGL_\Z}\to\caD_{\MGL_S}$ 
the pullback functor between $\MGL$-modules. 
It has a right adjoint $\RRR f_{\MGL,*}$.
As prior to Proposition \ref{base-change},
we conclude $\RRR f_{\MGL,*}$ preserves sums and is a $\caD_{\MGL_\Z,\caT}$-module functor.
The proof of Proposition \ref{base-change} shows $\LLL f_\MGL^*$ represents the Landweber theory over $S$.

By inferring  the analog of Lemma \ref{base-change-quasi-m} our claim about the quasi-multiplication is proven 
along the lines of the corresponding statement in Theorem \ref{landw-thm}.
\end{proof}

\section{Operations and cooperations} \label{op-coop}
Let $\AAA_*$ be a Landweber exact Adams graded $\MU_*$-algebra and $\EE$ a motivic spectrum with a 
quasi-multiplication which represents the corresponding Landweber exact theory.
Denote by $\EE^\Top$ the ring spectrum representing the corresponding topological Landweber exact theory. 
Then $\EE^\Top_* \cong \AAA_*$, 
$\EE^\Top$ is a commutative monoid in the stable homotopy category and there are no even degree nontrivial 
phantom maps between such topological spectra \cite[Section 2.1]{hov-strick.mor}.
\begin{proposition} 
\label{coop}
In the above situation the following hold.
\begin{itemize}
\item[(i)] 
$\EE_{**} \EE \cong \EE_{**} \otimes_{\EE^\Top_*} \EE^\Top_* \EE^\Top$.
\item[(ii)] 
$\EE$ satisfies the assumption of Corollary \ref{hopf-algebroids}$(ii)$.
\item[(iii)]  
The flat Hopf algebroid $(\EE_{**},\EE_{**}\EE)$ is induced from $(\MGL_{**},\MGL_{**} \MGL)$ via the map 
$\MGL_{**}\to\MGL_{**}\otimes_{\MU_*}\AAA_*\cong \EE_{**}$.
\end{itemize}
\end{proposition}
\begin{proof}
The isomorphism $\EE_{**} \FF \cong \MGL_{**} \FF \otimes_{\MU_*} \AAA_*$ can be recast as 
$$
\EE_{**} \FF \cong \MGL_{**} \FF \otimes_{\MGL_*} \MGL_* \otimes_{\MU_*} \EE^\Top_*
\cong \MGL_{**} \FF \otimes_{\MGL_*} \EE_*
$$
and
$$
\EE_{**} \FF \cong \MGL_{**} \FF \otimes_{\MGL_{**}} \MGL_{**} \otimes_{\MU_*}
\EE^\Top_* \cong \MGL_{**} \FF \otimes_{\MGL_{**}} \EE_{**}.
$$
In particular, 
$\EE_{**} \EE \cong \MGL_{**} \EE \otimes_{\MGL_{**}} \EE_{**}\cong \EE_{**} \MGL \otimes_{\MGL_{**}} \EE_{**}$
is isomorphic to 
\begin{equation} 
\label{induced-hopf}
(\MGL_{**} \MGL \otimes_{\MGL_{**}} \EE_{**}) \otimes_{\MGL_{**}} \EE_{**}\cong
\EE_{**} \otimes_{\MGL_{**}} \MGL_{**} \MGL \otimes_{\MGL_{**}} \EE_{**}.
\end{equation}
Moreover, since $\MGL_{**}\MGL\cong \MGL_{**}\otimes_{\MU_*}\MU_* \MU$, 
$$
\EE^\Top_* \otimes_{\MU_*} \MGL_{**} \MGL \otimes_{\MU_*} \EE^\Top_*
\cong \EE^\Top_* \otimes_{\MU_*} \MGL_{**} \otimes_{\MU_*} \MU_* \MU
\otimes_{\MU_*} \EE^\Top_*$$
is isomorphic to 
$$
\MGL_{**} \otimes_{\MU_*} \EE^\Top_* \EE^\Top\cong 
\MGL_{**} \otimes_{\MU_*} \EE^\Top_* \otimes_{\EE^\Top_*} \EE^\Top_* \EE^\Top\cong 
\EE_{**} \otimes_{\EE^\Top_*} \EE^\Top_* \EE^\Top.
$$
This proves the first part of the proposition.
In particular, 
\begin{equation} 
\label{coop-basech}
\EE_* \EE \cong \EE_* \otimes_{\EE^\Top_*} \EE^\Top_* \EE^\Top
\end{equation}
and 
\begin{equation} 
\label{coop-basech2}
\EE_{**} \EE \cong \EE_{**} \otimes_{\EE_*} \EE_* \EE.
\end{equation}
We note that $\EE^\Top_* \EE^\Top$ is flat over $\EE^\Top_*$ by the topological analog of (\ref{induced-hopf}) 
(this equation shows $\Spec(\EE^\Top_* \EE^\Top)= \Spec(\EE^\Top_*) \times_{\FG^s} \Spec(\EE^\Top_*)$).
Hence by (\ref{coop-basech}) $\EE_* \EE$ is flat over $\EE_*$.
Together with (\ref{coop-basech2}) this is Part $(ii)$ of the proposition.
Part $(iii)$ follows from (\ref{induced-hopf}).
\end{proof}
\begin{remark}\label{landweberandsmash}
Let $\EE^\Top$ and $\FF^\Top$ be evenly graded topological Landweber exact spectra, 
$\EE$ and $\FF$ the corresponding motivic spectra. 
Then $\EE \wedge \FF$ is Landweber exact corresponding to the $\MU_*$-module $(\EE^\Top \wedge \FF^\Top)_*$ 
(with either $\MU_*$-module structure).

\end{remark}
\begin{theorem} 
\label{stab-op-kgl}
\begin{itemize}
\item[(i)] The map afforded by the Kronecker product
\[ 
\xymatrix{
\KGL^{**}\KGL\ar[r] & \Hom_{\KGL_{**}}(\KGL_{**}\KGL,\KGL_{**}) }
\]
is an isomorphism of $\KGL^{**}$-algebras.
\item[(ii)] With the completed tensor product there is an isomorphism of $\KGL^{**}$-algebras
\[ 
\KGL^{**} \KGL\cong 
\KGL^{**}\widehat{\otimes}_{\KU^*}\KU^*\KU
\]
\end{itemize}
Item $(i)$ and the module part of $(ii)$ generalize to $\KGL^{**}(\KGL^{\wedge j})$ for $j>1$.
\end{theorem}
\begin{proof}
Recall $\KU_* \KU$ is free over $\KU_*$ \cite{adams-clarke} and $\KGL$ is the Landweber theory determined by 
the $\MU_*$-algebra $\MU_*\to\Z[\beta,\beta^{-1}]$ which classifies the multiplicative formal group law 
$\x+\y-\beta\x\y$ over $\Z[\beta,\beta^{-1}]$ with $|\beta|=2$ \cite[Theorem 1.2]{spitzweck-oestvaer}.
The corresponding topological Landweber exact theory is $\KU$ by the Conner-Floyd theorem.
Thus by Proposition \ref{coop} $(i)$  
$\KGL_{**}\KGL$ is free over $\KGL_{**}$.
Moreover, 
$\KGL$ has the structure of an $E_\infty$-motivic ring spectrum, 
see \cite{gepner-snaith}, \cite{spitzweck-oestvaer}, 
so the Universal coefficient spectral sequence in \cite[Proposition 7.7]{dugger-isaksen} can be applied to the 
$\KGL$-modules $\KGL \wedge \KGL$ and $\KGL$; it converges conditionally \cite{boardman}, \cite{mcclearly}, 
and with abutment $\Hom^{**}_{\KGL-\Mod}(\KGL \wedge \KGL, \KGL)=\Hom^{**}_{\SH(S)}(\KGL, \KGL)$.
But the spectral sequence degenerates since $\KGL_{**}\KGL$ is a free $\KGL_{**}$-module.
Hence items $(i)$ and $(ii)$ hold for $j=1$.

The more general statement is proved along the same lines by noting the isomorphism 
$$
\EE^\Top_*((\EE^\Top)^{\wedge j})\cong
\EE^\Top_* \EE^\Top \otimes_{\EE^\Top_*}\cdots\otimes_{\EE^\Top_*} \EE^\Top_* \EE^\Top,
$$ 
and similarly for the Adams graded and Adams graded graded motivic versions.
\end{proof}

In stable homotopy theory there is a universal coefficient spectral sequence for every Landweber exact ring theory 
\cite[Proposition 2.21]{hov-strick.mor}.
It appears there is no direct motivic analog:  
While there is a reasonable notion of evenly generated motivic spectrum as in \cite[Definition 2.10]{hov-strick.mor}
and one can show that a motivic spectrum representing a Landweber exact theory is evenly generated as in 
\cite[Proposition 2.12]{hov-strick.mor},
this does not have as strong consequences as in topology because the coefficient ring $\MGL_*$ is not concentrated 
in even degrees as $\MU_*$, 
but see Theorem \ref{phantoms} below.
We aim to extend the above results on homotopy algebraic $K$-theory to more general Landweber exact motivic spectra. 
\begin{proposition}
\label{spectral}
Suppose $\MMM$ is a Tate object and $\EE$ an $\MGL$-module.
Then there is a trigraded conditionally convergent right half-plane cohomological spectral sequence
\[ 
\EE_2^{a,(p,q)}=\Ext^{a,(p,q)}_{\MGL_{**}}(\MGL_{**}\MMM,\EE_{**})\Rightarrow \EE^{a+p,q}\MMM.
\]
\end{proposition}
\begin{proof} 
$\MGL\wedge\MMM$ is a cellular $\MGL$-module so this follows from \cite[Proposition 7.10]{dugger-isaksen}.
\end{proof}
The differentials in the spectral sequence go
\[ 
\xymatrix{
d_r\colon\EE_r^{a,(p,q)}\ar[r] & \EE_r^{a+r,(p-r+1,q)}. }
\]
\begin{theorem}
\label{maps}
Suppose $\MMM_*$ is a Landweber exact graded $\MU_*$-module concentrated in even degrees and $\MMM\in\SH(S)_\caT$ 
represents the corresponding motivic cohomology theory.
Then for $p,q\in\Z$ and $\NNN$ an $\MGL$-module spectrum there is a short exact sequence
\[ 
\xymatrix{
0\ar[r] &
\Ext^{1,(p-1,q)}_{\MGL_{**}}(\MGL_{**}\MMM,\NNN_{**})\ar[r] & 
\NNN^{p,q}\MMM\ar[r]^-{\pi} &
\Hom_{\MGL_{**}}^{p,q}(\MGL_{**}\MMM,\NNN_{**})\ar[r] & 0. }
\]
\end{theorem}
\begin{proof}
Let $\MMM^{\Top}$ be the topological spectrum associated with $\MMM_*$. 
Then $\MU_*\MMM^{\Top}$ is a flat $\MU_*$-module of projective dimension at most one,
see \cite[Propositions 2.12, 2.16]{hov-strick.mor}.
Hence $\MGL_{**}\MMM=\MGL_{**}\otimes_{\MU_*}\MU_*\MMM^{\Top}$ is a $\MGL_{**}$-module of projective dimension 
at most one and consequently the spectral sequence of Proposition \ref{spectral} degenerates at its $\EE_2$-page.
This implies the derived $\lim^1$-term $\lim^1\EE_r^{***}$ of the spectral sequence is zero; hence it converges strongly.
The assertion follows because $\EE_\infty^{p,**}=0$ for all $p\neq 0,1$.
\end{proof}
\begin{remark}
\label{blubb}
\begin{itemize}
\item[(i)] For $p,q\in\Z$, 
the group of phantom maps $\PPP^{p,q}(\MMM,\NNN)\subseteq\NNN^{p,q}\MMM$ is defined as 
$\{\Sphere^{p,q}\wedge \MMM\stackrel{\varphi}{\to}\NNN\, |\,
\mbox{ for all }\EE\in\SH(S)_{\caT,\f}\mbox{ and } \EE\stackrel{\nu}{\to}\Sphere^{p,q}\wedge \MMM: \varphi\nu=0\}$.
It is clear that $\PPP^{p,q}(\MMM,\NNN)\subseteq\mathrm{ker}(\pi)$.
\item[(ii)] 
The following topological example due to Strickland shows a nontrivial $\Ext^1$-term. 
The canonical map $\KU_{(p)}\to \KU_p$ from $p$-local to $p$-complete unitary topological $K$-theory yields a 
cofiber sequence
\[ 
\xymatrix{
\KU_{(p)}\ar[r] &  \KU_p\ar[r] & \EE \ar[r]^-{\delta} & \Sigma \KU_{(p)}. }
\]
Here $\EE$ is rational and thus Landweber exact. Thus $\delta$
is a degree $1$ map between Landweber exact spectra.

However, $\delta$ is a nonzero phantom map. 

Over fields embeddable into $\C$ the corresponding boundary map
for the motivic Landweber spectra is likewise phantom and non-zero.
Using the notion of heights for Landweber exact algebras from \cite[Section 5]{niko}, 
observe that $\EE$ has height zero while $\Sigma \KU_{(p)}$ has height one, 
compare with the assumptions in Theorem \ref{phantoms} below.
\end{itemize}\end{remark}

Now fix Landweber exact $\MU_*$-algebras $\EE_*$ and $\FF_*$ concentrated in even degrees and a $2$-commutative diagram
\begin{equation}
\label{commutative}
\xymatrix{
\Spec(\FF_*)\ar[rr]^f\ar[rd]_-{f_\FF} & & \Spec(\EE_*)\ar[dl]^-{f_\EE}\\ 
& \caX & }
\end{equation}
where $\caX$ is the stack of formal groups and $f_\FF$ (resp.~$f_\EE$) the map classifying the formal group $G_\FF$ 
(resp.~$G_\EE$) canonically associated with the complex orientable cohomology theory corresponding to $\FF_*$ (resp.~$\EE_*$). 
This entails an isomorphism $f^*G_\EE\cong G_\FF$ of formal groups over $\Spec(\FF_*)$.
Hence the height of $\FF_*$ is less or equal to the height of $\EE_*$. 
Let $\EE^{\Top},\FF^{\Top}$ (resp.~$\EE,\FF\in\SH(S)_\caT$) be the topological (resp.~motivic) spectra representing the 
indicated Landweber exact cohomology theory.
\begin{theorem}
\label{phantoms}
With the notation above assume $\EE_*^{\Top}\EE^{\Top}$ is a projective $\EE_*^{\Top}$-module. 
\begin{itemize}
\item[(i)] 
The map from Theorem \ref{maps}
\[ 
\xymatrix{
\pi:\FF^{**}\EE\ar[r] & 
\Hom_{\MGL_{**}}^{**}(\MGL_{**}\EE,\FF_{**})\cong
\Hom_{\EE_*^{\Top}}(\EE_*^{\Top}\EE^{\Top},\FF_{**}) }
\]
is an isomorphism.
\item[(ii)] 
Under the isomorphism in $(i)$, 
the bidegree $(0,0)$ maps $\Sphere^{*,*}\wedge \EE\to \FF$ which respect the quasi-multiplication correspond bijectively to 
maps of $\EE_*^{\Top}$-algebras
\[ 
\Hom_{\EE_*^{\Top}-\Alg}(\EE_*^{\Top}\EE^{\Top},\FF_{**}).
\]
\end{itemize}
\end{theorem}

\begin{remark}
\label{land-rem}
\begin{itemize}
\item[(i)] The assumptions in Theorem \ref{phantoms} hold when $\EE^{\Top}=\KU$ and for certain localizations of 
Johnson-Wilson theories according to \cite{adams-clarke} respectively \cite{baker}. 
Theorem \ref{phantoms} recovers Theorem \ref{stab-op-kgl} with no mention of an $E_\infty$-structure on $\KGL$.
\item[(ii)] The theorem applies to the quasi-multiplication
$(\EE \wedge \EE \to \EE)\in\EE^{00}(\EE\wedge\EE)$ and shows that this is a commutative monoid structure which lifts 
uniquely the multiplication on the homology theory.
For example, 
there is a unique structure of commutative monoid on $\KGL_S\in\SH(S)$ representing the familiar multiplicative 
structure of homotopy $K$-theory, see \cite{PPR1} for a detailed account and an independent proof in the case $S=\Spec(\Z)$.
\item[(iii)] 
The composite map $\alpha:\EE_*\stackrel{f}{\to}\FF_*\to\MGL_{**}\otimes_{\MU_*}\FF_*=\FF_{**}$ yields a canonical bijection
between the sets $\Hom_{\EE_*^{\Top}-\Alg}(\EE_*^{\Top}\EE^{\Top},\FF_{**})$ and $\{(\alpha',\varphi)\}$, 
where $\alpha'\colon\EE_*\rightarrow\FF_{**}$ is a ring homomorphism and $\varphi\colon\alpha_*G_\EE\to\alpha_*'G_\EE$
a strict isomorphism of strict formal groups.
\item[(iv)] 
Taking $\FF=\EE$ in Theorem \ref{phantoms} and using Remark \ref{blubb}$(i)$ implies that $\PPP^{**}(\EE,\EE)=0$. 
For example,
there are no nontrivial phantom maps $\KGL\to\KGL$ of any bidegree.
\end{itemize}
\end{remark}
\begin{proof} (of Theorem \ref{phantoms}):
We shall apply Proposition \ref{ext} with $X_0\equiv \Spec(\MU_*)$, $X\equiv \Spec(\FF_*)$, 
$Y\equiv \Spec(\EE_*)$, $f_X\equiv f_\FF$ and $f_Y\equiv f_\EE$, 
$\pi:\Spec(\MU_*)\to\caX$ the map classifying the universal formal group, $f$ as given by (\ref{commutative}) 
and $\alpha:X=\Spec(\FF_*)\to X_0=\Spec(\MU_*)$ corresponding to the $\MU_*$-algebra structure $\MU_*\to \FF_*$. 
Now by \cite[Theorem 26]{niko}, $f_X$ (resp.~$f_Y$) factors as $f_X=i_X\circ\pi_X$ (resp.~$f_Y=i_Y\circ\pi_Y$) 
with $\pi_X$ and $\pi_Y$ faithfully flat and $i_X$ and $i_Y$ inclusions of open substacks. 
The map $i$ in Proposition \ref{ext} is induced by $f$.
Finally,
$\MGL_{**}$ is canonically an $\MU_*\MU$-comodule algebra and the $\caO_\caX$-algebra $\caA$ in 
Proposition \ref{ext} corresponds to $\MGL_{**}$, i.e.~$\caA(X_0)=\MGL_{**}$ and 
$\pi_Y^*\pi_{Y,*}\caO_Y\in\Qc_Y$ to the projective $\EE_*^{\Top}$-module $\EE_*^{\Top}\EE^{\Top}$.
Taking into account the isomorphisms
\[ 
\caA(X_0)\otimes_{\caO_{X_0}}\pi^*f_{Y,*}\caO_Y\cong\MGL_{**}\otimes_{\MU_*} \MU_*^{\Top}\EE^{\Top}\
\cong\MGL_{**}\EE
\]
\[ 
\caA(X_0)\otimes_{\caO_{X_0}}\alpha_*\caO_X\cong\MGL_{**}\otimes_{\MU_*}\FF^{\Top}_*\cong \FF_{**}
\]
\[
\pi^*_Y\pi_{Y,*}\caO_Y\cong \EE^{\Top}_*\EE^{\Top}
\]
\[ 
\caA(Y)\otimes_{\caO_Y}f_*\caO_X\cong \FF_{**}
\]
\[ 
\caO_Y\cong \EE_*^{\Top}
\]
we obtain from Proposition \ref{ext}
\[ 
\Ext^n_{\MGL_{**}}(\MGL_{**}\EE,\FF_{**})\cong\left\{ \begin{array}{lcc} 0  & n\geq 1,\\
\Hom_{\EE_*^{\Top}}(\EE_*^{\Top}\EE^{\Top},\FF_{**}) &  n=0.\end{array}\right. 
\]
Hence $(i)$ follows from Theorem \ref{maps} and $(ii)$ by unwinding the definitions.
\end{proof}

\section{The Chern character}\label{CHERN}
In what follows we define a ring map from $\KGL$ to periodized rational motivic cohomology which induces the Chern 
character (or regulator map) from $K$-theory to (higher) Chow groups when the base scheme is smooth over a field.

Let $\MZ$ denote the integral motivic Eilenberg-MacLane ring spectrum introduced by Voevodsky 
\cite[\S6.1]{voevodsky-icm}, 
cf.~\cite[Example 3.4]{motivicfunctors}.
Next we give a canonical orientation on $\MZ$, 
in particular the construction of a map $\P^\infty_+ \to K(\Z(1),2)=L((\P^1,\infty))$. 

Recall the space $L(X)$ assigns to any $U$ the group of proper relative cycles on 
$U\times_S X$ over $U$ of relative dimension $0$ which have universally integral coefficients.
Now the line bundle $\OO_{\P^n}(1) \boxtimes \OO_{\P^1}(n)$ acquires the section 
$l_n\equiv T_n x_0^n + T_{n-1} x_0^{n-1} x_1 + \cdots + T_0 x_1^n$, 
where $[T_0:\cdots:T_n]$ denotes homogeneous coordinates on $\P^n$ and $[x_0:x_1]$ on $\P^1$. 
Its zero locus is a relative divisor of degree $n$ on $\P^1$ which induces a map $\P^n \to L(\P^1)$.
These maps combine to give maps $\P^n \to L((\P^1,\infty))$ which are compatible with the inclusions $\P^n \to \P^{n+1}$. 
Hence there is an induced map $\varphi\colon \P^\infty \to K(\Z(1),2)$.
Moreover, 
the map $\P^n \to L(\P^1)$ is additive with respect to the maps $\P^n \times \P^m \to \P^{n+m}$ induced by 
multiplication by the section $l_n$. 
Hence $\varphi$ is a map of commutative monoids and it restricts to the canonical map $\P^1 \to K(\Z(1),2)$. 
This establishes an orientation on $\MZ$ with the additive formal group law.

Let $\MQ$ be the rationalization of $\MZ$.
In order to apply the spectral sequence in Proposition \ref{spectral} to $\MQ$ we equip it with an $\MGL$-module structure.
Note that both $\MZ$ and $\MQ$ have canonical $E_\infty$-structures.
Thus $\MQ \wedge \MGL$ is also $E_\infty$. 
As an $\MQ$-module it has the form $\MQ[b_1,b_2,\ldots]$. 
For any generator $b_i$ we let $\iota_i \colon \Sigma^{2i,i} \MQ \to \MQ \wedge \MGL$ denote the corresponding map. 
Taking its adjoint provides a map from the free $\MQ$-$E_\infty$-algebra on $\bigvee_{i > 0} \Sphere^{2i,i}$ to 
$\MQ\wedge \MGL$. 
Since we are dealing with rational coefficients the contraction of these cells in $E_\infty$-algebras is isomorphic to $\MQ$. 
Hence there is a map $\MGL \to \MQ$ in $E_\infty$-algebras.
This gives in particular an $\MGL$-module structure on $\MQ$.

Let $\PMQ$ be the periodized rational Eilenberg-MacLane spectrum considered as an $\MGL$-module,
and $\LQ$ the Landweber spectrum corresponding to the additive formal group law over $\Q$.
By Remark \ref{land-rem} $\LQ$ is a ring spectrum.
We let $\PLQ$ be the periodic version.
Both $\LQ$ and $\PLQ$ have canonical structures of $\MGL$-modules.
Finally, 
let $\PHQ$ be the periodized rational topological Eilenberg-MacLane spectrum.

Recall the map $\Ch^{\PH}_* \colon \KU_* \to \PHQ_*$ sending the Bott element to the canonical element in degree $2$.
The exponential map establishes an isomorphism from the additive formal group law over $\PHQ_*$ to the pushforward
of the multiplicative formal group law over $\KU_*$ with respect to $\Ch^{\PH}_*$.
By Theorem \ref{phantoms} and Remark \ref{land-rem}$(iii)$ there is an induced map of motivic ring spectra 
$\Ch^{\PL}\colon \KGL \to \PLQ$.
\begin{theorem}
The rationalization 
\begin{equation*}
\xymatrix{
\Ch^{\PL}_{\Q}\colon \KGL_\Q \ar[r] & \PLQ }
\end{equation*}
of the map $\Ch^{\PL}$ from $\KGL$ to $\PLQ$ is an isomorphism.
\end{theorem}
\begin{proof}
Follows directly from the fact that the rationalization of $\Ch^{\PH}_*$ is an isomorphism.
\end{proof}

Theorem \ref{maps} shows there is a short exact sequence
\[ 
\xymatrix{
0\ar[r] &
\Ext^{1,(p-1,q)}_{\MGL_{**}}(\MGL_{**}\LQ,\MQ_{**}) \ar[r] & 
\MQ^{p,q}\LQ \\
\ar[r]^-{\pi} &
\Hom_{\MGL_{**}}^{p,q}(\MGL_{**}\LQ,\MQ_{**})\ar[r] & 0. }
\]
Now since $\MQ$ has the additive formal group law there is a natural transformation of homology theories
\begin{equation}
\label{equation:LQtoMQ}
\xymatrix{
\LQ_{**}(-) \ar[r] &  \MQ_{**}(-). }
\end{equation} 
Applying the methods of Theorem \ref{phantoms} to $\EE=\LQ$ and $\FF=\MQ$ shows that (\ref{equation:LQtoMQ})
lifts uniquely to a map of motivic ring spectra
\[ 
\xymatrix{
\iota \colon \LQ \ar[r] & \MQ. }
\] 
It prolongs to a map of motivic ring spectra $\PLQ \to\PMQ$ (denoted by the same symbol).

The composite map
\[ 
\xymatrix{
\Ch^{\PM}\colon\KGL\ar[r]^-{\Ch^{\PL}} & \PLQ \ar[r]^-{\iota} & \PMQ }
\] 
is called the Chern character.
By construction it is functorial in the base scheme with respect to the natural map $\LLL f^* \PMQ_S \to \PMQ_{S'}$ 
for $f\colon S' \to S$.

Recall that for smooth schemes over fields motivic cohomology coincides with higher Chow groups \cite{voevodsky-allagree}.

\begin{proposition}
Evaluated on smooth schemes over fields the map $\Ch^{\PM}$ coincides with the
usual Chern character from $K$-theory to higher Chow groups.
\end{proposition}
\begin{proof}
The construction of the Chern character in \cite{bloch} and \cite{levine} uses the methods of \cite{gillet}.
We first show that the individual Chern class transformations $C_i$ in loc.~cit.~from $K$-theory 
to the cohomology theory in question can be extended to a transformation between simplicial presheaves
on smooth affine schemes over the given field $k$. 
Fix a cofibration
\begin{equation*}
\xymatrix{
\mathrm{BGL}(\Z) \ar[r] & \mathrm{BGL}^+(\Z).} 
\end{equation*}
The simplicial presheaf 
\begin{equation*}
\Spec(A)\mapsto 
\Gamma(A)\equiv
\Z\times\mathrm{BGL}(A)\cup_{\mathrm{BGL}(\Z)} \mathrm{BGL}^+(\Z)
\end{equation*}
represents $K$-theory, see \cite{bloch}. 
The Chern class $C_i$ of the universal vector bundle on the sheaf $\mathrm{BGL}(-)$ can be represented
by a transformation of simplicial presheaves $\mathrm{BGL}(-) \to K(i)$, 
where $K(i)$ denotes an injectively fibrant presheaf of simplicial abelian groups representing 
motivic cohomology with coefficients in $\Q(i)$ with the appropriate simplicial shift.
The map $\mathrm{BGL}(k) \to K(i)(k)$ extends to 
\begin{equation*}
\xymatrix{
\Gamma(k) \ar[r] &  K(i)(k). } 
\end{equation*}
By definition of the presheaf $\Gamma$ we get the required map.
Having achieved this, 
the Chern class transformations $C_i$ extend to functors on the full subcategory ${\mathcal F}$ of 
objects of finite type in the sense of \cite{voevodsky-icm} in the $\A^1$-local homotopy category. 
Denote by $j\colon {\mathcal F}\to\SH(k)$ the canonical functor. 

With the above observations as prelude, 
it follows that these transformations induce a multiplicative Chern character transformation
\begin{equation*}
\xymatrix{
\tau \colon \Gamma(-) \ar[r] & \PMQ_{00}(-) \circ j } 
\end{equation*}
on this category.
The source and target of $\tau$ are $\PP^1$-periodic and $\tau$ is compatible with these.
Hence there is an induced transformation on the Karoubian envelope of
the Spanier-Whitehead stabilization with respect to the pointed $\PP^1$, 
which is the full subcategory of $\SH(k)$ of compact objects according to \cite[Propositions 5.3 and 5.5]{voevodsky-icm}.
But as a cohomology theory on compact objects, 
$\KGL$ is the universal oriented theory which is multiplicative for the formal group law. 
To conclude the proof, 
it is now sufficient to note that the transformation constructed above has the same effect on the universal first 
Chern class as $\Ch^{\PM}$ does, which is clear.
\end{proof}

For smooth quasi-projective schemes over fields the Chern character is known to be an isomorphism after rationalization \cite{bloch},
hence our transformation $\Ch^{\PM}$ is an isomorphism after rationalization
(a map $\EE \to \FF$ between periodic spectra is an isomorphism if it induces isomorphisms $\EE^{-i,0}(X) \to \FF^{-i,0}(X)$ 
for all smooth schemes $X$ over $S$ and $i \ge 0$).
By Mayer-Vietoris the same holds for smooth schemes over fields.
\begin{corollary}\label{lqismq}
For smooth schemes over fields the map 
\[ 
\xymatrix{
\iota \colon \LQ \ar[r] & \MQ }
\] 
is an isomorphism of motivic ring spectra.
\end{corollary}
\begin{corollary}
For smooth schemes over fields $$\MQ_{**}(-)$$ is the universal oriented homology theory with rational coefficients and 
additive formal group law.
\end{corollary}

Next we identify the rationalization $\MGL_\Q$ of the algebraic cobordism spectrum:
\begin{theorem}
\label{splitmgl}
There are isomorphisms of motivic ring spectra
\[ 
\MGL_\Q\cong\MGL\wedge\LQ\cong\LQ[b_1,\ldots],
\]
where the generator $b_i$ has bidegree $(2i,i)$ for every $i\geq 1$.
\end{theorem}
\begin{proof} 
According to Remark \ref{landweberandsmash} $\MGL\wedge\LQ$ is the motivic Landweber exact spectrum associated with
$\MU\wedge\HQ\cong\MU_\Q$; this implies the first isomorphism.
In homotopy, 
the canonical map of ring spectra $\LQ\to\MGL\wedge\LQ$ yields
\[ 
\xymatrix{ \pi_{**}\LQ=\MGL_{**}\otimes_{\MU_*}\Q\ar[r] & 
\pi_{**}(\MGL\wedge\LQ)=\MGL_{**}\MGL\otimes_{\MU_*}\Q=\pi_{**}\LQ[b_1,\ldots].  }
\]
Hence there is a map of ring spectra $\LQ[b_1,\ldots]\to\MGL\wedge\LQ$ under $\LQ$ which is an $\pi_{**}$-isomorphism. 
Since all spectra above are cellular the second isomorphism follows.
\end{proof}
\begin{corollary}
\label{mglrational} 
Suppose $S$ is smooth over a field.
\begin{itemize}\item[(i)] 
There are isomorphisms of motivic ring spectra 
\[ 
\MGL_\Q\cong\MGL\wedge\MQ\cong\MQ[b_1,\ldots].
\]
\item[(ii)] For $X/S$ smooth and $\LL^*$ the (graded) Lazard ring, there is an isomorphism 
\[ 
\MGL^{**}(X)\otimes_\Z\Q\cong\MQ^{**}(X)\otimes_\Z\LL^*.
\]
\end{itemize}
\end{corollary}
\begin{proof}
Part (i) is immediate from Theorem \ref{splitmgl}, 
specialized to smooth schemes over fields, 
and Corollary \ref{lqismq}. 
Part (ii) follows from (i) using compactness of $X$.
\end{proof}

As alluded to in the introduction we may now explicate the rationalized algebraic cobordism of number fields.
The answer is conveniently formulated in terms of the (graded) Lazard ring $\LL^*=\Z[x_1,x_2,\ldots]$ with its cohomological 
grading $|x_i|=(-2i,-i)$, $i\geq 1$.
\begin{corollary}
\label{mglrationalnumberfield}
Suppose $k$ is a number field with $r_1$ real embeddings and $r_2$ pairs of complex embeddings. 
Then there are isomorphisms 
\begin{equation*}
\MGL^{2i,j}(k)\otimes\Q\cong
\begin{cases}
\LL^{2i}\otimes\Q & j=i\\
0 & j\neq i
\end{cases}
\end{equation*}
\begin{equation*}
\MGL^{2i+1,j}(k)\otimes \Q\cong
\begin{cases}
\LL^{2i}\otimes k^*\otimes \Q & j=i+1, i\leq 0\\
\LL^{2i}\otimes\Q^{r_2} & j-i\equiv 3\, (4),j-i>1\\
\LL^{2i}\otimes\Q^{r_1+r_2} & j-i\equiv 1\, (4),j-i>1 \\
0 & \text{otherwise.} 
\end{cases}
\end{equation*}
\end{corollary}
\begin{proof} 
Follows from Corollary \ref{mglrational}(ii) and the well-known computation of the rational 
motivic cohomology of number fields.
\end{proof}

\begin{remark} 
In Corollary \ref{lqismq} we identified the (unique) Landweber exact motivic spectrum 
$\LQ$ with rational motivic cohomology $\MQ$ (for base schemes smooth over some field).
The topological analog of this result is a triviality because $\HQ$ is the Landweber exact 
spectrum associated with the additive formal group over $\Q$.
To appreciate the content of Corollary \ref{lqismq}, 
we offer the following remark: 
In stable homotopy theory it is trivial that $S^0_\Q\cong \HQ$ but the motivic analog of 
this result fails.
Let $\unit_{\Q}$ denote the rationalized motivic sphere spectrum. 
Using orthogonal idempotents, 
Morel \cite{morelrationalsphere} has constructed a splitting
\[ 
\unit_{\Q}\cong \unit_{\Q}^+ \vee \unit_{\Q}^-
\]
and noted that $\unit_{\Q}^-$ is nontrivial for formally real fields 
(e.g.~the rational numbers).
It is easy to show that every map from the motivic sphere spectrum to an oriented motivic 
ring spectrum annihilates $\unit_{\Q}^-$. 
In particular, 
$\unit_{\Q}$ and $\LQ$ are not isomorphic in general.
\end{remark}

\begin{center}
Fakult{\"a}t f{\"u}r Mathematik, Universit{\"a}t Regensburg, Germany.\\
e-mail: niko.naumann@mathematik.uni-regensburg.de
\end{center}
\begin{center}
Fakult{\"a}t f{\"u}r Mathematik, Universit{\"a}t Regensburg, Germany.\\
e-mail: Markus.Spitzweck@mathematik.uni-regensburg.de
\end{center}
\begin{center}
Department of Mathematics, University of Oslo, Norway.\\
e-mail: paularne@math.uio.no
\end{center}
\end{document}